\documentclass[a4paper,11pt]{article}%
\usepackage{amsmath}%
\usepackage{amsfonts}%
\usepackage{amssymb}%
\usepackage{amsthm}
\usepackage{stmaryrd}
\usepackage{graphicx}
\usepackage[utf8]{inputenc}
\usepackage[T1]{fontenc}
\usepackage[french, english]{babel}
\usepackage{enumitem}
\usepackage{tikz}
\usetikzlibrary{decorations.markings}
\usepackage[colorlinks=true]{hyperref}

\newtheorem{thm}{Theorem}

\newtheorem{theorem}{Theorem}

\newtheorem{claim}[theorem]{Claim}

\newtheorem{corollary}[theorem]{Corollary}

\newtheorem{lemma}[theorem]{Lemma}

\newtheorem{proposition}[theorem]{Proposition}

\theoremstyle{definition}
\newtheorem{definition}[theorem]{Definition}

\newtheorem{remark}[theorem]{Remark}

\newlength{\espaceavantspecialthm}
\newlength{\espaceapresspecialthm}
\newcommand{\R}{\mathbb{R}}
\newcommand{\N}{\mathbb{N}}

\newcommand{\Z}{\mathbb{Z}}
\newcommand{\T}{\mathbb{T}}
\newcommand{\Sp}{\mathbb{S}}

\newcommand{\Pj}{\mathbb{P}}
\newcommand{\varep}{\varepsilon}

\newcommand{\Homeo}{\operatorname{Homeo}}

\newcommand{\Id}{\operatorname{Id}}

{\\}

\newenvironment{defi*}[1][]{
\vskip \espaceavantspecialthm \noindent \textbf{D\'efinition.} }%
{\vskip \espaceapresspecialthm}

\counterwithin{equation}{section}

\tikzset{->-/.style={decoration={
  markings,
  mark=at position .5 with {\arrow{latex}}},postaction={decorate}}}

\newcommand\test[1]{
\pgfmathsetmacro{\var}{#1}
\pgfmathparse{ifthenelse(\var>=0,"positif","négatif")} \pgfmathresult}%

\newcommand{\hgline}[3]{
\pgfmathsetmacro{\thetaone}{#1}
\pgfmathsetmacro{\thetatwo}{#2}
\pgfmathsetmacro{\theta}{(\thetaone+\thetatwo)/2}
\pgfmathsetmacro{\phi}{abs(\thetaone-\thetatwo)/2}
\pgfmathsetmacro{\close}{less(abs(\phi-90),0.0001)}
\ifdim \close pt = 1pt
    \draw[->-, color=#3] (\thetaone:1) -- (\thetatwo:1);
\else
	\pgfmathsetmacro{\R}{tan(\phi)}
	\pgfmathsetmacro{\test}{(\thetaone-\thetatwo)/abs(\thetaone-\thetatwo)}
	\draw[->-, color=#3] (\thetaone:1) arc (\thetaone+\test*90:\thetaone+\test*(270-2*\phi):\R);
\fi
}

\newcommand{\hglinefill}[3]{
\pgfmathsetmacro{\thetaone}{#1}
\pgfmathsetmacro{\thetatwo}{#2}
\pgfmathsetmacro{\theta}{(\thetaone+\thetatwo)/2}
\pgfmathsetmacro{\phi}{abs(\thetaone-\thetatwo)/2}
\pgfmathsetmacro{\close}{less(abs(\phi-90),0.0001)}
\ifdim \close pt = 1pt
    \filldraw[->-, color=#3] (\thetaone:1) -- (\thetatwo:1) ;
\else
	\pgfmathsetmacro{\R}{tan(\phi)}
	\pgfmathsetmacro{\test}{(\thetaone-\thetatwo)/abs(\thetaone-\thetatwo)}
	\filldraw[color=#3, opacity=.2] (\thetaone:1) arc (\thetaone+\test*90:\thetaone+\test*(270-2*\phi):\R) arc (\thetaone+\test*(270-2*\phi)-90:\thetaone+\test*90+90:1);
\fi
}

\begin{document}

\sloppy

\title{Parabolic isometries of the fine curve graph of the torus}
\author{Pierre-Antoine Guih\'eneuf, Emmanuel Militon}
\maketitle

\begin{abstract}
In this article we finish the classification of actions of torus homeomorphisms on the fine curve graph initiated by Bowden, Hensel, Mann, Militon, and Webb in \cite{BHMMW}. 

This is made by proving that if $f \in \mathrm{Homeo}(\T ^2)$, then $f$ acts elliptically on $C^{\dagger}(\T ^2)$ if and only if $f$ has bounded deviation from some $v \in \mathbb{Q}^2 \setminus \left\{ 0 \right\}$. The proof involves some kind of slow rotation sets for torus homeomorphisms.
\end{abstract}

\selectlanguage{english}

\setcounter{tocdepth}{1}
\tableofcontents

\section{Introduction}

The \emph{fine curve graph} $C^\dagger(S)$ of a closed surface $S$ was introduced by Bowden, Hensel, and Webb \cite{BHW} to give a counterpart of the classical \emph{curve graph} adapted to the study of the group of all homeomorphisms of $S$. More precisely, the classical curve graph $C(S)$ has vertex set the isotopy classes of essential simple closed curves on S, with edges between pairs of isotopy classes that can be realized disjointly (a slight modification is needed for genus 1 surfaces). Note that the natural action of a homeomorphism on curves quotients down to an action of the mapping class group $\mathrm{Map}(S)$ on the curve graph $C(S)$ by isometries.
The Gromov hyperbolicity (or equivalently $\delta$-hyperbolicity) of $C(S)$, showed by Masur and Minsky \cite{zbMATH01355494, zbMATH01545126}, then implies numerous geometric and algebraic properties of the mapping class group $\mathrm{Map}(S)$ (e.g. \cite{zbMATH06521340, zbMATH06035994, zbMATH05578711, zbMATH06751842, zbMATH05868010}).

In this paper we will focus on the case of the torus $S = \T^2$. Let us give the precise definition of the fine curve graph in this context.

\begin{definition}
The \emph{fine curve graph} on the torus $\T^2$ is the graph $C^\dagger(\T^2)$ whose vertices are essential\footnote{\emph{I.e.} non contractible.} simple loops. There is an edge between two vertices $\alpha$ and $\beta$ if and only if the loops $\alpha$ and $\beta$ have at most one intersection point.
\end{definition}

As a consequence of the Gromov hyperbolicity of the classical curve graphs for punctured surfaces, it was proved in \cite{BHW} that the fine curve graph $C^\dagger(S)$ is Gromov hyperbolic. This enables the authors to use large scale geometry techniques to study Homeo(S) via its action on $C^\dagger(S)$. As an application, they prove that, for any closed surface $S$ of genus $\ge 1$, stable commutator length and fragmentation norm on $\mathrm{Homeo}_0(S)$ are unbounded, answering a question posed by Burago, Ivanov, and Polterovich \cite{zbMATH05526532}.
\medskip

In the same way as the mapping class group $\mathrm{Map}(S)$ acts on $C(S)$ by isometries, the whole homeomorphism group $\mathrm{Homeo}(S)$ acts on $C^\dagger(S)$ by isometries. Gromov has classified isometries of Gromov hyperbolic spaces \cite[paragraph 8]{zbMATH04031953}, \cite{zbMATH01385418}, according to the asymptotic translation length, defined for an isometry $g$ of a Gromov hyperbolic space $X$ as
\[|g|_X = \lim_{n\to+\infty}\frac{1}{n} d_X\big(x,g^n(x)\big).\]
It is a standard exercise to see that this limit exists and is independent of $x$. This independence immediately implies that the asymptotic translation length is a conjugacy invariant of isometries of $X$.
Gromov classification is then as follows: for $g$ an isometry of a Gromov hyperbolic space, $g$ is
\begin{itemize}
\item \textbf{Hyperbolic} if the asymptotic translation length is positive;
\item \textbf{Parabolic} if the asymptotic translation length is zero but $g$ has no
finite diameter orbits, and
\item \textbf{Elliptic} if $g$ has finite diameter orbits.
\end{itemize}
There is an equivalent reformulation of this trichotomy in terms of fixed points on the Gromov boundary of $X$, but we do not require this point of view in the present work.

While there is no mapping class acting parabolically on\footnote{It follows from \cite[Proposition 4.6]{zbMATH01355494} and the Nielsen–Thurston classification \cite{zbMATH04103989}.} $C(S)$, the situation is much richer for the action of homeomorphisms on $C^\dagger(S)$: in \cite{BHMMW}, the authors prove that there are homeomorphisms of $\T^2$ acting parabolically on $C^\dagger(\T^2)$. They also initiate a classification of actions of homeomorphisms on $C^\dagger(\T^2)$ in terms of rotational behaviour: they give a criterion of hyperbolicity in terms of rotation set, as well as examples of parabolic and elliptic homeomorphisms.
In the present article, we complete their work to give a complete classification of actions of homeormorphisms of $\T^2$ in terms of rotational behaviour.

\begin{definition}
Let $v \in \R^2 \setminus \left\{ 0 \right\}$. We say that $f\in\Homeo(\T^2)$ has \emph{bounded deviation from direction $v$} if there exists $\rho\in\R^2$ and a lift $\tilde{f}: \mathbb{R}^2 \rightarrow \mathbb{R}^2$ of $f$ such that $|\langle\tilde{f}^n(x)-x-n\rho,v\rangle|$ is bounded from above, uniformly in $x\in\R^2$ and $n\in\N$.
\end{definition}

Observe that this definition does not depend on the chosen lift $\tilde{f}$ of $f$.

Recall the classification of homotopy classes for torus homeomorphisms: if $f\in\Homeo(\T^2)$, then $f$ of $f^2$ is homotopic either to the identity, or to a Dehn twist (defined in Section~\ref{SecDefRot}), or to a linear Anosov automorphism.

A homeomorphism homotopic to a linear Anosov automorphism has unbounded deviation in any direction, and a homeomorphism homotopic to a Dehn twist has unbounded deviation in any direction but possibly one. To see this, consider one point of $\R^2$ and some nontrivial integer translate of it, and look at the deviation between these two points (at some point one has to use the fact that the eigendirections of linear Anosov automorphisms have irrational slope).

Note that if $f\in \mathrm{Homeo}(\T^2)$ is homotopic to identity and has bounded deviation from two non-collinear directions, then it has bounded deviation from any direction.
\bigskip

The main theorem of this article is the following.

\begin{thm} \label{ThmElliptictorus}
Let $f \in \mathrm{Homeo}(\T ^2)$. Then $f$ acts elliptically on $C^{\dagger}(\T ^2)$ if and only if $f$ has bounded deviation from some $v \in \mathbb{Q}^2 \setminus \left\{ 0 \right\}$.
\end{thm}

Combined with the results of \cite{BHMMW}, this theorem implies a complete classification of the action of homeomorphisms on $C^{\dagger}(\T ^2)$ in terms of rotational behaviour.

We denote by $\mathrm{Homeo}_0(\T ^2)$ the connected component of identity in the group $\mathrm{Homeo}(\T ^2)$; it coincides with the set of homeomorphisms that are homotopic to identity.
We denote the rotation set of $\tilde f$ by $\rho(\tilde{f})$. It is a compact convex subset of the plane capturing the rotational behaviour of the homeomorphism (see Section~\ref{SecDefRot} for precise definitions).

\begin{corollary}\label{CoroMain}
Let $f \in \mathrm{Homeo}_0(\T ^2)$. Then 
\begin{itemize}
\item $f$ acts hyperbolically on $C^{\dagger}(\T ^2)$ if and only if $\rho(\tilde f)$ has nonempty interior;
\item $f$ acts parabolically on $C^{\dagger}(\T ^2)$ if and only if $\rho(\tilde f)$ is a segment of irrational direction, or $\rho(\tilde f)$ is a segment of rational direction not passing through a rational point, or $f$ is a pseudo-rotation\footnote{We call \emph{pseudo-rotation} a homeomorphism $f\in\Homeo_0(\T^2)$ whose rotation set is reduced to a single point.} with unbounded deviation from any $v \in \mathbb{Q}^2 \setminus \left\{ 0 \right\}$;
\item $f$ acts elliptically on $C^{\dagger}(\T ^2)$ if and only if $\rho(\tilde f)$ is a segment of rational direction passing through a rational point, or $f$ is a pseudo-rotation with bounded deviation from some $v \in \mathbb{Q}^2 \setminus \left\{ 0 \right\}$.
\end{itemize}
\end{corollary}

Let us explain how to deduce this corollary from Theorem~\ref{ThmElliptictorus}.
The first point is \cite[Theorem 1.3]{BHMMW}. Hence it suffices to distinguish homeomorphisms that act elliptically from those that do not.

By Passeggi and Sambarino \cite{zbMATH07228227}, any homeomorphism whose rotation set is a segment of rational slope not passing by a rational point, has unbounded deviation, and hence by Theorem~\ref{ThmElliptictorus} acts parabolically.

Moreover, by D\'avalos \cite{zbMATH06914177}, any homeomorphism whose rotation set is a segment of rational direction $v$ passing through a rational point has bounded deviation in $v^\perp$, and hence by Theorem~\ref{ThmElliptictorus} acts elliptically.
 
If we suppose that the Franks-Misiurewicz conjecture \cite{zbMATH04149169} holds, the above corollary implies the following improvement of the second point: $f$ acts parabolically on $C^{\dagger}(\T ^2)$ if and only if $\rho(\tilde f)$ is a segment of irrational direction, or $f$ is a pseudo-rotation with unbounded deviation from some $v \in \mathbb{Q}^2 \setminus \left\{ 0 \right\}$. In particular any $f$ whose rotation set is a segment with rational direction acts elliptically on $C^{\dagger}(\T ^2)$.

Remark that there are both pseudo-rotations with bounded displacement in some rational direction (like actual rotations) and pseudo-rotations with unbounded displacement in any rational direction (see for \cite{MR3238423} for rational pseudo-rotations and \cite{zbMATH05613068} for irrational ones). Hence one cannot distinguish completely the possible types of actions of homeomorphisms on $C^\dagger(\T^2)$ only in terms of rotation sets.
Note that Jäger \cite{MR2501297}, Jäger-Tal \cite{zbMATH06782951} and Kocsard \cite{zbMATH07391928} gives criteria of semi-conjugation to a rotation for torus homeomorphisms with bounded displacements.

One can give a statement similar to Corollary~\ref{CoroMain} from the rotation viewpoint:

\begin{corollary}
Let $f \in \mathrm{Homeo}_0(\T ^2)$. Then 
\begin{itemize}
\item If $\rho(\tilde f)$ has nonempty interior, then $f$ acts hyperbolically on $C^{\dagger}(\T ^2)$;
\item If $\rho(\tilde f)$ is a segment with irrational slope, then $f$ acts parabolically on $C^{\dagger}(\T ^2)$;
\item If $\rho(\tilde f)$ is a segment with rational slope passing through a rational point, then $f$ acts elliptically on $C^{\dagger}(\T ^2)$;
\item If $\rho(\tilde f)$ is a segment with rational slope not passing through a rational point (a case that should never hold according to Franks-Misiurewicz conjecture \cite{zbMATH04149169}), then $f$ acts parabolically on $C^{\dagger}(\T ^2)$;
\item If $f$ is a pseudo-rotation with unbounded deviation from any $v \in \mathbb{Q}^2 \setminus \left\{ 0 \right\}$, then $f$ acts parabolically on $C^{\dagger}(\T ^2)$;
\item If $f$ is a pseudo-rotation with bounded deviation from some $v \in \mathbb{Q}^2 \setminus \left\{ 0 \right\}$, then $f$ acts elliptically on $C^{\dagger}(\T ^2)$.
\end{itemize}
\end{corollary}

Theorem~\ref{ThmElliptictorus} also allows to give a complete classification of the action of homeomorphisms on $C^{\dagger}(\T ^2)$ for the ones that are not homotopic to identity. 

As proved in \cite{BHMMW}, any homeomorphism having an iterate homotopic to an Anosov linear automorphism acts hyperbolically on $C^\dagger(\T^2)$. Thus it remains to classify the actions in the case of a homeomorphism having an iterate homotopic to a Dehn twist.

\begin{corollary}
Let $f \in \mathrm{Homeo}(\T ^2)$ such that $f^r$ is homotopic to a Dehn twist.
\begin{itemize}
\item $f$ acts hyperbolically on $C^{\dagger}(\T ^2)$ if and only if $\rho(\tilde f^r)$ has nonempty interior;
\item $f$ acts parabolically on $C^{\dagger}(\T ^2)$ if and only if $\rho(\tilde f^r)$ is a single number, and $f$ has unbounded displacement in the vertical direction;
\item $f$ acts elliptically on $C^{\dagger}(\T ^2)$ if and only if $\rho(\tilde f^r)$ is a single number, and $f$ has bounded displacement in the vertical direction.
\end{itemize}
\end{corollary}

Note that by Addas-Zanata, Tal, and Garcia \cite{zbMATH06296542}, if $\rho(\tilde f^r)$ is reduced to a single rational number, then $f$ has bounded displacement in the vertical direction and hence acts elliptically on $C^{\dagger}(\T ^2)$. To our knowledge, the question whether the second case is nonempty (\emph{i.e.} if there exists homeomorphisms homotopic to Dehn twists acting parabolically on $C^{\dagger}(\T ^2)$) is open.

Remark that by \cite{zbMATH06272683}, on a open and dense subset of $\Homeo_0(\T^2)$, the rotation set is a polygon with rational vertices, hence the set of parabolic elements of $\Homeo_0(\T^2)$ is included in a closed set with empty interior.

A last comment: in \cite{arxiv.2210.05460} Le Roux and Wolff prove that any automorphism of a variant of the fine curve graph is realized by some homeomorphism of the surface. They suggest that the same result holds for our definition of the fine curve graph, hence in some sense our classification of actions of homeomorphisms on the fine curve graph covers the whole automorphism group of the ifne curve graph.

\paragraph*{Some open questions}
\begin{enumerate}
\item Can we get a similar classification for higher genus surfaces\footnote{The authors have a strategy for a characterization of homeomorphisms isotopic to identity acting hyperbolically on $C^\dagger(S_g)$: they should be the ones with nonempty interior homological rotation set, and the ones with a pseudo-Anosov mapping class when removing some periodic orbit. It should be the subject of a future work.}? 
\item Are there some torus homeomorphisms homotopic to identity satisfying $\liminf d_n<+\infty$ and $\limsup d_n = +\infty$, where $d_n$ is the image under a lift of $f^{n}$ to $\mathbb{R}^2$ of a fixed fundamental domain?
\item Are there some torus homeomorphisms homotopic to identity acting parabolically but not properly on $C^\dagger(\T^2)$?
\item More generally, what are the sets of possible good limit values? Can they be classified?
\end{enumerate}

A potential master's student of the first author should start thinking about the last three questions soon.

\subsection*{Acknowledgments} 

The first author was supported by a PEPS-JCJC grant.
The second author was supported by the ANR project Gromeov ANR-19-CE40-0007.

The authors warmly thank Alejandro Kocsard and Roberta Shapiro for their useful comments about the first version of this article.

\subsection{Rotation sets for torus homeomorphisms}\label{SecDefRot}

Let $f\in \mathrm{Homeo}_0(\T ^2)$ and fix a lift $\tilde{f}: \mathbb{R}^2 \rightarrow \mathbb{R}^2$.

The \emph{rotation set} of $\tilde f$ is the set
\[\rho(\tilde f) = \left\{v\in\R^2\ \Big\vert\  \exists (x_k)_k\in\R^2, (n_k)_k\to+\infty : \frac{\tilde f^{n_k}(x_k)-x_k}{n_k}\underset{k\to+\infty}{\longrightarrow} v\right\}.\]
A theorem of Misiurevicz and Ziemian \cite{MR1053617} states that it is a compact convex subset of $\R^2$ (see also Lemma~\ref{LemConvexlimit}).
Some basic properties are straightforward consequences of the definition: for any $k\in\Z$, $\rho(\tilde f^k) = k\rho(\tilde f)$, and $\rho$ is a conjugacy invariant: if $g\in \mathrm{Homeo}_0(\T ^2)$ and $\tilde g$ is a lift of $g$ to $\R^2$, then $\rho(\tilde g \tilde f\tilde g^{-1}) = \rho(\tilde f)$.
It depends on the lift of $f$ in the following way: any other lift of $f$ can be written $\tilde f+v$, with $v\in\Z^2$. Then $\rho(\tilde f+v) = \rho(\tilde f)+v$.

We will be inspired by an equivalent formulation of the rotation set in Section~\ref{SecGood} to define good limit values: fix a fundamental domain $D\subset\R^2$ of the torus (\emph{e.g.} $D = [0,1]^2$), then (see \cite{MR1053617})
\[\rho(\tilde f) = \lim_{n\to+\infty} \frac{\tilde f^n(D)}{n},\]
where the limit holds in the Hausdorff topology (and in particular, the result states that the limit does exist).
\bigskip

Now, let $f\in \mathrm{Homeo}(\T ^2)$ homotopic to a Dehn twist. By this, we mean that there exists a basis of the torus in which $f$ is homotopic to the linear automorphism
\[\begin{pmatrix}
1 & k \\ 0 & 1
\end{pmatrix}\]
for some $k\in\Z^*$.

Fix a lift $\tilde{f}: \mathbb{R}^2 \rightarrow \mathbb{R}^2$, and denote by $p_2 : \R^2\to\R$ the projection on the second coordinate (according to the basis used to define the Dehn twist).

Following \cite{zbMATH01049390}, we can define the \emph{rotation set} of $\tilde f$ as
\[\rho(\tilde f) = \left\{v\in\R\ \Big\vert\  \exists (x_k)_k\in\R^2, (n_k)_k\to+\infty : p_2\Big(\frac{\tilde f^{n_k}(x_k)-x_k}{n_k}\Big) \underset{k\to+\infty}{\longrightarrow} v\right\}.\]
This set is a segment of $\R$. As for homeomorphisms isotopic to the identity, it follows from the definition that for any $g\in\Homeo_0(\T^2)$, we have $\rho(\tilde f) = \rho(\tilde g\tilde f\tilde g^{-1})$. Moreover, two lifts of $f$ to $\R^2$ have rotation sets which differ by an integral translation of $\R$.

\subsection{Outline}

One way of Theorem~\ref{ThmElliptictorus} is easy to prove: if the homeomorphism $f$ has bounded deviation from a rational direction, then $f$ acts elliptically on $C^{\dagger}(\T ^2)$. We write down the proof of this implication.

\begin{proof}[Proof of the "if" part in Theorem \ref{ThmElliptictorus}] Let us say that $f$ has bounded deviation from $v \in \mathbb{Q}^{2} \setminus \left\{ 0 \right\}$. Without loss of generality, we can suppose that $v=(p,q)$, where either $p$ and $q$ are relatively prime integers, or $(p,q)\in\{(0,1),(1,0)\}$. Let $\alpha : \mathbb{R}/\mathbb{Z} \rightarrow \T ^2$ be the loop defined by $t \mapsto t(q,-p)$.
When $\gamma$ and $\gamma'$ are isotopic loops of $\T ^2$, we denote by $ C_{\gamma}(\gamma')$ the number of lifts of $\gamma$ that are met by a given lift $\tilde{\gamma}'$ of $\gamma'$. This number does not depend on the chosen lift $\tilde{\gamma}'$ of $\gamma'$. 

Note that (as already said before) a homeomorphism homotopic to a linear Anosov automorphism has unbounded deviation in any direction, and a homeomorphism homotopic to a Dehn twist has unbounded deviation in any direction but possibly one. In the last case, this direction is such that the loops $(f^n(\alpha))_{n\in\N}$ are all homotopic one to each other.

As $f$ has bounded deviation from direction $v=(p,q)$, there exists $\rho\in\R^2$ and a lift $\tilde{f}: \mathbb{R}^2 \rightarrow \mathbb{R}^2$ of $f$ such that $|\langle\tilde{f}^n(x)-x-n\rho,v\rangle|$ is bounded uniformly in $x$ and $n$. 
This implies that the sequence $(C_{\alpha+n \rho}(f^{n}(\alpha))_{n \geq 0}$ is bounded, where $\alpha+n \rho$ is the loop $t\mapsto \alpha(t)+n \rho$.
Hence, as the loop $\alpha+n\rho$ is either disjoint from $\alpha$ or equal to $\alpha$, the sequence $(C_{\alpha}(f^{n}(\alpha)))_{n \geq 0}$ is bounded. But, by Lemma 4.5 of \cite{BHMMW}, for any $n \geq 0$,
\[C_{\alpha}(f^{n}(\alpha)) +1 \geq d_{C^{\dagger}(\T ^2)}(\alpha,f^{n}(\alpha))\]
and the orbit of $\alpha$ under $f$ in $C^{\dagger}(\T ^2)$ is bounded.
\end{proof}

To prove the direct implication, we suppose that $f$ has unbounded deviation from any rational direction; we want to prove that, in this case, $f$ does not act elliptically on $C^{\dagger}(\T ^2)$.

The first step is to define \emph{good limit values} (Section~\ref{SecGood}), that are analogs of rotation sets capturing sublinear speeds: instead of dividing by the time, one divides by the diameter of the iterate of the fundamental domain. An important property is that good limit values are convex, as stated in Lemma~\ref{LemConvexlimit}.

We then state a (non) ellipticity criterion in terms of good limit values. This criterion (Proposition~\ref{PropParabolicityCriterion}) improves a criterion given in \cite{BHMMW} and relies on branched covering maps of the torus by square tiled surfaces.

The next step is done in Section~\ref{SecPossible}. We set a dichotomy for possible shapes of good limit values: either one of them contains a segment of irrational slope, or there exists a rational line containing any good limit value. In the first case, the criterion set in the previous section applies and shows that the action of the homeomorphism is non elliptic. It then remains to treat the second case.

The final argument is given in Section~\ref{SecFinal}. We show that if any good limit value is contained in a single rational line, and if $f$ has unbounded deviation from any rational direction, then (a modified version of) some good limit value contains a segment of irrational slope. This allows to apply once again the parabolicity criterion.

\section{Good limit values}\label{SecGood}

In this section we define good limit values, that could also be called ``slow rotation sets'': they capture the rotational behaviour in the case when the rotation speed is sublinear. 

For any $n \geq 0$, we denote by $d_n$ the diameter of $\tilde{f}^n(D)$, where $D$ is the fundamental domain $[0,1]^2$ of the torus $\T^2$. The following lemma implies that there is a subsequence of  $(d_n)_n$ which tends to $+\infty$. In the sequel we will need a more precise result which is the second part of this lemma.

\begin{lemma}\label{LemDiamDeviation}
If $f\in\Homeo(\T^2)$ is not homotopic to a linear Anosov automorphism and has unbounded deviation in some direction, then $$\sup_{n\in\N} d_n=+\infty.$$

More precisely, if there exists $v\in\R^2$, $\rho\in\rho(\tilde f)$, $(n_k)$ going to infinity, and a sequence of points $(x_k)$ such that for any $k\in\N$, $|\langle\tilde{f}^{n_k}(x_k)-x_k-n_k\rho,v\rangle| \ge k$, then
\begin{equation}\label{EqDiamDeviation}
\lim_{k\to+\infty}\, \sup_{x,y\in D} \big|\big\langle \tilde f^{n_k}(x)-\tilde f^{n_k}(y),v\big\rangle\big| =+\infty.
\end{equation}
\end{lemma}

Let us explain this statement. If $f\in\Homeo(\T^2)$ has no iterate homotopic to identity, then $\sup_{n\in\N} d_n=+\infty$, so the first part is relevant only in the case $f\in\Homeo_0(\T^2)$. If $f$ is homotopic to a Dehn twist about the horizontal direction, then \eqref{EqDiamDeviation} holds for any $v\notin \R(0,1)$, so the lemma is only relevant in the case $v\in\R (0,1)$, in which case one can define
\[\big\langle\tilde{f}^{n_k}(x_k)-x_k-n_k\rho,v\big\rangle :=  p_2\big(\tilde{f}^{n_k}(x_k)-x_k\big)-n_k\rho.\]

In the case where $f$ is homotopic to a linear Anosov automorphism, then \eqref{EqDiamDeviation} holds for any direction $v$, except possibly the stable direction of the automorphism (which is irrational).

\begin{proof}
We prove the lemma in the case $f\in\Homeo_0(\T^2)$, the case where $f$ is homotopic to a Dehn twist is identical.

Suppose that $f$ has unbounded deviation in some direction $v$. In particular, there exists $\rho\in\rho(\tilde f)$, $(n_k)$ going to infinity, and a sequence of points $(x_k)$ such that for any $k$, $|\langle\tilde{f}^{n_k}(x_k)-x_k-n_k\rho,v\rangle| \ge k$. Taking the image of each $x_k$ under an integral translation if necessary, we can suppose that the points $x_k$ all belong to $D$. Without loss of generality, by taking a subsequence if necessary, we can suppose that each of those scalar products are positive. 

If the set $\{\langle \rho,v\rangle\mid \rho\in\rho(\tilde f)\}$ is a nontrivial segment, then the conclusion of the lemma is straightforward. If not, then the quantity of \eqref{EqDiamDeviation} does not depend on the choice of $\rho\in\rho(\tilde f)$.

Suppose that the last limit of the lemma does not hold. Then, by considering a subsequence if necessary, it holds that
\[\sup_{k\in\N}\, \sup_{x,y\in D^2} \big|\langle \tilde f^{n_k}(x)-\tilde f^{n_k}(y),\,v\rangle\big| =R<+\infty.\]
Then for any $y\in D$, applying this to $x=x_k$, we have (considering that $n_0=0$)
\begin{align*}
\langle\tilde{f}^{n_k}(y)-y-n_k\rho,v\rangle = & \,\langle\tilde{f}^{n_k}(y) - \tilde{f}^{n_k}(x_k),v\rangle + \langle \tilde{f}^{n_k}(x_k)-x_k-n_k\rho,v\rangle\\
& + \langle x_k-y,v\rangle\\
\ge &\, k-2R.
\end{align*}
Observe that the left side of the above inequality does not change if we replace the point $y$ by one of its integral translate so that the inequality actually holds for any $y \in \mathbb{R}^2$.  
By choosing $k\ge 2R+1$, we get $\langle\tilde{f}^{n_k}(y)-y-n_k\rho,v\rangle \ge 1$. As this holds for any $y\in \R^2$, one can iterate: for any $z\in\R^2$ and any $\ell\in\N$, we have $\langle\tilde{f}^{\ell\,n_k}(z)-z-\ell\,n_k\rho,v\rangle \ge \ell$. In particular, this implies that there exists $\rho'\in\rho(\tilde f)\setminus (\rho+\R v^\perp)$, a contradiction.
\end{proof}

Fix $f\in\Homeo(\T^2)$ and a point $\tilde{x}_0 \in \mathrm{int}(D)$.
For any subset $A$ of $\mathbb{R}^2$, any $\lambda >0$ and any $v \in \mathbb{R}^2$, let us denote $ \lambda A +v = \left\{ \lambda a +v \mid a \in A \right\}$.
For any $n \geq 0$, let 
$$A_n= \frac{1}{d_n}\Big(\tilde{f}^n(D) -\tilde{f}^{n}(\tilde{x}_0)\Big).$$

More generally, fix a sequence $(a_k)_{k \geq 0}$ of positive real numbers as well as a sequence $(n_k)_{k \geq 0}$ of integers with $n_k \rightarrow +\infty$. For any $k \geq 0$, define 
\begin{equation}\label{EqDefBk}
B_k= \frac{1}{a_k} \left( \tilde{f}^{n_{k}}(D)- \tilde{f}^{n_k}(x_0) \right).
\end{equation}

As $\tilde{x}_0 \in D$, observe that, for any $n \geq 0$, $0 \in A_n$. Moreover, by definition of $(d_n)$, for any $n \geq 0$, the set $A_n$ is contained in the closed unit disc of $\mathbb{R}^2$ and has diameter 1. Recall that the set of compact subsets of the closed unit disc, endowed with the Hausdorff topology, is compact.

We endow the set of closed subsets of $\R^2$ with the following topology, which resembles Hausdorff convergence on any (large) ball.
Let $\phi$ be the map that to any closed subset $F$ of $\R^2$ associates the compact subset $F\cup\{\infty\}$ of the Alexandroff compactification of $\R^2$ (which is homotopic to $\Sp^2$). The topology on closed subsets of $\R^2$ we consider is then the initial topology associated to $\phi$ and Hausdorff topology on the Alexandroff compactification of $\R^2$.
The set of closed subsets of $\R^2$ endowed with this topology is compact.

\begin{definition}\label{DefGood}
We call \emph{good limit value} of the sequence $(A_{n})$ any limit value $A_\infty$ of this sequence such that there exists a subsequence $(A_{n_k})_{k \geq 0}$ which converges to $A_\infty$ with $\displaystyle \lim_{k \rightarrow + \infty} d_{n_k}= +\infty$.
\end{definition}

Note that a good limit value of the sequence $(A_n)$ is a limit value of the sequence $(B_k)$ associated to some sequence $(n_k)$ and with $a_k=d_{n_k}$.

By Lemma~\ref{LemDiamDeviation} --- combined with the compactness of the set of compact subsets of the unit disk endowed with Hausdorff topology --- if $f$ has unbounded deviation in some direction, then it has at least one good limit value. More generally, for any sequences $(a_k)_{k \geq 0}$ of positive real numbers and $(n_k)_{k \geq 0}$ of integers with $n_k \rightarrow +\infty$, the sequence $(B_k)$ admits a limit value (not necessarily compact).

The following lemma is a direct adaptation of a result of Misiurevicz and Ziemian \cite{MR1053617} asserting that rotation sets are convex.

\begin{lemma} \label{LemConvexlimit}
Suppose that $\lim_{ k \rightarrow +\infty} a_k = +\infty$. Then any limit value $B_{\infty}$ of the sequence $(B_k)$ is a convex subset of $\R^2$.
\end{lemma}

Note that this implies that any good limit value $A_{\infty}$ of the sequence $(A_n)$ is a convex subset of the closed unit disc.

\begin{proof}
Let $\xi_1$ and $\xi_2$ be two points of a limit value $B_\infty$ of the sequence $(B_k)_k$. Then, extracting a subsequence if necessary, there exist sequences $(p_k)_k$ and $(q_k)_k$ of points of $D$ such that 
$$ \lim_{k \rightarrow + \infty} \frac{1}{a_k} \Big(\tilde{f}^{n_k}(p_{k})- \tilde{f}^{n_k}({x}_0)\Big)= \xi_1$$
and
$$ \lim_{k \rightarrow + \infty} \frac{1}{a_k} \Big(\tilde{f}^{n_k}(q_{k})- \tilde{f}^{n_k}({x}_0)\Big)= \xi_2.$$
Let $ \xi= \lambda \xi_1 +(1-\lambda) \xi_2$, with $0< \lambda <1$ and let us prove that $\xi \in B_{\infty}$. For any $k \geq 0$, let $z_k= \lambda \tilde{f}^{n_k}(p_k)+(1-\lambda) \tilde{f}(q_k)$. By Lemma 3.3 of the article \cite{MR1053617} by Misiurewicz and Ziemia, the set $\tilde{f}^{n_k}(D)$ is $\sqrt{2}$-quasi-convex, so that for any $k \geq 0$ there exists a point $r_k \in D$ such that $d(\tilde{f}^{n_k}(r_k),z_k) \leq \sqrt{2}$. Then the sequence 
$$\bigg(\frac{1}{a_k} \Big(\tilde{f}^{n_k}(r_k)-\tilde{f}^{n_k}({x}_0)\Big)\bigg)_k$$
has the same limit as the sequence
$$\bigg(\frac{1}{a_k} \Big(z_k-\tilde{f}^{n_k}(\tilde{x}_0)\Big)\bigg)_k$$
and this limit is $\xi$. Hence the point $\xi$ belongs to $B_\infty$ and the set $B_\infty$ is convex.  
\end{proof}

\section{A non-ellipticity criterion}

Let $f\in\Homeo(\T^2)$ and $x_0 \in \T ^2$. The goal of this section is to prove the following criterion, which generalizes Section 6 of \cite{BHMMW}.

\begin{proposition}[Criterion of non-ellipticity] \label{PropParabolicityCriterion}
Suppose that the following holds:
\begin{enumerate}
\item $\displaystyle \lim_{ k \rightarrow +\infty} a_k = +\infty$.
\item There exists a sequence $(w_k)_k$ of vectors of $\R^2$ such that the sequence
\[(B_k+w_k)_{k \geq 0} = 
\left(\frac{1}{a_k} \left( \tilde{f}^{n_{k}}(D)- \tilde{f}^{n_k}(x_0) \right)+w_k\right)_{k\ge 0}\]
of compact subsets of $\mathbb{R}^2$ converges to some closed subset $B_{\infty}$ for the topology defined before Definition~\ref{DefGood}.
\item The set $B_{\infty}$ contains a nontrivial segment of irrational slope.
\end{enumerate}
Then $f$ does not act as an elliptic isometry of $C^{\dagger}(\T ^2)$.
\end{proposition}

Recall that by \cite[Theorem 1.3]{BHMMW}, $f$ acts hyperbolically on $C^{\dagger}(\T ^2)$ if and only if $\operatorname{int}(\rho(\tilde f))\neq\emptyset$. Hence this criterion can be used to prove that some homeomorphism acts parabolically on $C^\dagger(\T^2)$.

To prove this proposition, we need to introduce some notation and three lemmas.

For any integer $m>0$ we let
$$\T ^2_{m}= \mathbb{R} / m \mathbb{Z} \times \mathbb{R} / m \mathbb{Z}.$$
This space can be seen as a cover of $\T ^2=\mathbb{R}^2 / \mathbb{Z}^2$ of degree $m^2$ via the projection $\T^2_{m} \rightarrow \T ^2$. We denote by $p_{m} : \mathbb{R}^2 \rightarrow \T ^2_{m}$ the projection.

We endow $\T ^2_{m}$ with the translation surface structure which makes the map
$$ \begin{array}{rcl}
\T ^2=\mathbb{R}/\mathbb{Z} \times \mathbb{R} / \mathbb{Z} & \rightarrow & \T ^2_{m} \\
(x,y) & \mapsto & (mx,my)
\end{array}$$
an isomorphism, where $\T ^2$ is endowed with the usual translation surface structure. This means that, in comparison to the usual euclidean metric on $\mathbb{R}^2$, distances are multiplied by $\frac{1}{m}$ on both coordinates (hence the diameter of $\T_m^2$ is $\sqrt 2$).

The first lemma we need is a purely topological lemma. We state it in the case of the torus, which is the case we need, but it is valid on any surface.

\begin{lemma} \label{LemTopology}
Let $\gamma_1$ and $\gamma_2$ be two simple paths $[0,1] \rightarrow \T ^2$ which are homotopic with fixed extremities in $\T ^2$. Then there exists a nonempty open set $U \subset \T ^2$ such that, for any point $p$ of $U$, the two paths $\gamma_1$ and $\gamma_2$ are homotopic with fixed extremities in $\T ^2 \setminus \left\{ p \right\}$.
\end{lemma}

\begin{proof}
Define the equivalence relation on $\T ^2$ whose equivalence classes are the singletons of points outside $\gamma_1$, and $\gamma_1$. This is the equivalence relation that shrinks the path $\gamma_1$ to a point. We denote by $T_1$ the quotient of $\T ^2$ by this equivalence relation and by $q_1$ the point that is the image of $\gamma_1$ in $T_1$.  

The space $T_1$ is still a $2$-torus and the image $\alpha_1$ of $\gamma_2$ in $T_1$ is a path. As $\gamma_2$ is simple, the path $\alpha_1$ has only self intersections at the point $q_1$ and the autointersection points cannot be transverse. Hence the path $\alpha_1$ is a homotopically trivial loop; it is a union of simple loops based at $q_1$ which do not meet each other except at $q_1$, some of them homotopically trivial and some of them homotopically non trivial (and there is a finite number of such last ones).

The complement of any homotopically trivial simple loop has one component which is homeomorphic to a disk, which we call the \emph{interior} of such a loop. Take the closure $C$ of the union of the interiors of the homotopically trivial simple loops appearing in the decomposition of $\alpha_1$. Observe that $C$ contains the point $q_1$, we can shrink $C$ to a point $q_2$ to obtain a new torus $T_2$. 

We call $\alpha_2$ the image of the path $\alpha_1$ in the torus $T_2$. The path $\alpha_2$ is the concatenation of a finite number of homotopically non trivial simple loops. Moreover, the loop $\alpha_2$ has auto-intersections only at the point $q_2$, none of which are transverse (too see this, use the fact that $\gamma_1$ is a simple loop). We then write $\alpha_2$ as a concatenation of loops $\beta_1,\beta_2, \ldots, \beta_r$, each of which is a homotopically trivial loop which cannot be written as a concatenation of nontrivial homotopically trivial loops. Fix a lift $\tilde{q}_2$ of the point $q_2$ to $\mathbb{R}^2$ and denote by $\tilde{\alpha}_2, \tilde{\beta}_1,\tilde{\beta}_2, \ldots, \tilde{\beta}_r$ the respective lifts of $\alpha_2,\beta_1,\beta_2, \ldots, \beta_r$ based at $\tilde{q}_2$. Observe that the loop $\tilde{\alpha}_2$ is the concatenation of the loops $\tilde{\beta}_1,\tilde{\beta}_2, \ldots, \tilde{\beta}_r$ and that each of the latter loop is simple. Observe that the interior of each $\tilde{\beta}_i$ is disjoint from its translates under the group of deck transformations. Hence the interior of $\tilde{\beta}_i$ projects injectively to the torus $T_2$ (because no transverse autointersection is allowed and no deck transformation has a fixed point). We call this projection the interior of the loop $\beta_i$, and we denote it by $I_i$. Observe also that, for $i \neq j$, either $I_i$ is contained in $I_j$ or $I_i$ contains $I_j$ or $I_i$ and $I_j$ are pairwise disjoint. An induction on $r$ shows that the complement of the union of the closures of the $I_i$'s is nonempty. It suffices to take as a set $U$ the preimage in $\T ^2$ of this subset of $T_2$ to prove the lemma.
\end{proof}

\begin{remark}
There is a shorter argument for the proof of this lemma, which uses the folkloric \emph{dual function}: the curve $\gamma_1\gamma_2^{-1}$ is homologous to 0. This allows to define a dual function on the complement of its image in the torus, and any point $p$ in which the dual function is equal to 0 suits the conclusion of the lemma.
\end{remark}

The three lemmas that follow are essentially proved in \cite{BHMMW}.

\begin{lemma} \label{LemLift}
Let $\alpha$ and $\beta$ be two curves in $C^{\dagger}(\T ^2)$ and $m \geq 1$. Then, for any respective lifts $\tilde{\alpha}$ and $\tilde{\beta}$ of $\alpha$ and $\beta$ in $\T ^2_m$ (via the cover map $\T^2_{m} \rightarrow \T ^2$), we have
$$d_{C^{\dagger}(\T ^2_m)}(\tilde{\alpha},\tilde{\beta}) \leq d_{C^{\dagger}(\T ^2)}(\alpha,\beta).$$
\end{lemma}

\begin{proof} 
This is a straightforward consequence of Lemma 6.3 in \cite{BHMMW}. 
\end{proof}

\begin{lemma} \label{LemGoodCover}
Fix $K>0$. There exists a square-tiled surface $\Sigma(K)$ such that, for any $m>0$ and any $p \in \T ^2_{m}$, there exists a branched covering map $f_{m,p}: \Sigma(K) \rightarrow \T ^2_{m}$, which is branched only at $p$, with the following properties.
\begin{enumerate}
\item For any two essential simple closed curves $\alpha$ and $\beta$ of $\T ^2_{m}$ with $d_{C^{\dagger}(\T ^2_{m})}(\alpha,\beta) \leq K$ and which do not meet the point $p$, there exist lifts of $\alpha$ and $\beta$ to $\Sigma(K)$ which are disjoint.
\item The map $f_{m,p}$ is a local isomorphism of translation surfaces outside $f_{m,p}^{-1}(\{p\})$.
\end{enumerate}
\end{lemma}

\begin{proof}
The proof is almost identical to the proof of Lemma 6.4 in \cite{BHMMW}.
\end{proof}

The following lemma is Lemma 6.5 in \cite{BHMMW}.

\begin{lemma} \label{LemIrrationalLines}
Let $\xi \in \mathbb{R} \setminus \mathbb{Q}$ and $\Sigma$ be a square-tiled surface. There exists $L' >0$ such that any line segment in $\Sigma$ of slope $\xi$ and of length greater than $L'$ meets any horizontal closed curve.
\end{lemma}

As a consequence of the above Lemma, we obtain the following corollary.

\begin{corollary} \label{CorIrrationalLines}
Fix $\xi$ and $\Sigma$ as in the above lemma and take $L=2L'$, where $L'$ is given by the above lemma. Then there exists $\varepsilon >0$ such that the following property holds. For any line segment $S$ which is $\varepsilon$-close, for the Hausdorff distance, to a line segment of slope $\xi$  and of length greater than $L$, any path which is homotopic to $S$ with fixed extremities in the complement of singular points in $\Sigma$ meets any horizontal closed curve.
\end{corollary}

\begin{proof} 
Consider the set $\mathcal{K}$ consisting of compact connected subsets of $\Sigma(K)$ which are either segments of length $L$ and slope $\xi$ or a union of two segments of $\Sigma$ of slope $\xi$ with a common extremity which is a singularity of $\Sigma$ and whose total length (\emph{i.e.} the sum of the lengths of those segments) is equal to $L$. By Lemma \ref{LemIrrationalLines}, any element of $\mathcal{K}$ meets any horizontal closed curve of $\Sigma$. Moreover, the set $\mathcal{K}$ is compact for the Hausdorff topology.

For any element $S'$ of $\mathcal{K}$, there exists $\varepsilon_{S'} >0$ such that any line segment $S$ which is $\varepsilon_{S'}$-close to $S'$ meets any horizontal curve.  By compactness of $\mathcal{K}$, we can take $\varepsilon_{S'}= \varepsilon$ independent of $S'$. Observe that any such segment $S$ as above has a nonzero algebraic intersection number with any horizontal curve so that Corollary \ref{CorIrrationalLines} holds. 
\end{proof}

\begin{proof}[Proof of Proposition \ref{PropParabolicityCriterion}]
We prove this proposition by contradiction. Suppose that $f$ acts elliptically on $C^{\dagger}(\T ^2)$. We denote by $\alpha$  the curve $t \mapsto (t,0)$ on $\T ^2$ and by $\tilde{\alpha}$ its lift $t \mapsto (t,0)$ to $\mathbb{R}^2$. Then there exists $K'>0$ such that, for any $n \geq 0$
$$ d_{C^{\dagger}(\T ^2)}\big(\alpha,f^{n}(\alpha)\big) \leq K'.$$
Let $K=K'+1$ and observe that, for any essential simple closed curve $\beta$ which is disjoint from $\alpha$, 
$$ d_{C^{\dagger}(\T ^2)}\big(\alpha,f^{n}(\beta)\big) \leq K.$$

Apply Lemma \ref{LemGoodCover} to obtain a square tiled surface $\Sigma(K)$. Take $\theta \in \mathbb{R} \setminus \mathbb{Q}$ such that the set $B_{\infty}$ contains  a nontrivial segment with irrational slope $\theta$. Apply Corollary~\ref{CorIrrationalLines} with this slope and the surface $\Sigma(K)$.
This corollary gives a number $L>0$. Fix $M>0$ in such a way that the set $M B_{\infty}$ contains an irrational segment $S_\infty$ with length $>L$ and slope $\theta$. Take a compact subset $C$ whose interior contains this segment $S_\infty$. Then the sequence of subsets
$$ \left( \frac{1}{\lfloor \frac{a_k}{M} \rfloor}\Big(\tilde{f}^{n_k}(D)-\tilde{f}^{n_k}(x_0)+M \Big\lfloor \frac{a_{k}}{M} \Big\rfloor w_{k}\Big) \right)_k $$
converges to the subset $M B_\infty$ for the Hausdorff topology. Take an integer $k_0$ sufficiently large so that the set
$$\frac{1}{\lfloor \frac{a_{k_0}}{M} \rfloor}\Big(\tilde{f}^{n_{k_0}}(D)-\tilde{f}^{n_{k_0}}(x_0)+M \big\lfloor \frac{a_{k_0}}{M} \big\rfloor w_{k_0}\Big) \cap C$$ 
is $\varepsilon$-close to the set $MB_\infty \cap C$, where $\varepsilon $ is given by Corollary \ref{CorIrrationalLines} and observe that the set
$$\frac{1}{\lfloor \frac{a_{k_0}}{M} \rfloor}\big(\tilde{f}^{n_{k_0}}(D)-\tilde{f}^{n_{k_0}}(x_0)\big) \cap (C-M w_{k_0})$$
is $\varepsilon$-close to the set $M(B_\infty-w_{k_0}) \cap (C-M w_{k_0})$. Fix $m=\lfloor \frac{a_{k_0}}{M} \rfloor$. 

Hence there exist two points $x,y \in D$ such that the point $\tilde{f}^{n_{k_0}}(x)-\tilde{f}^{n_{k_0}}(x_0)$ is $m\varepsilon$-close to one extremity of the segment $m(S_\infty-Mw_{k_0})$ and the point $\tilde{f}^{n_{k_0}}(y)-\tilde{f}^{n_{k_0}}(x_0)$ is $m\varepsilon$-close to the other extremity of this segment. Moreover, we choose those points $x$ and $y$ outside any lift of $\alpha$. 
In particular, the point $p_{m}(\tilde{f}^{n_{k_0}}(x)-\tilde{f}^{n_{k_0}}(x_0))$ is $\varepsilon$-close to one extremity of the segment $S=p_{m}(m(S_\infty-Mw_{k_0}))$ and the point $p_{m}(\tilde{f}^{n_{k_0}}(y)-\tilde{f}^{n_{k_0}}(x_0))$ is $\varepsilon$-close to the other extremity of this segment.
Observe that the length of the segment $S$ is greater than $L$. Denote by $S'$ the line segment in $\T ^2_{m}$ joining these two points that remains $\varepsilon$-close to $S$ and let $S''=S'+p_{m}(\tilde{f}^{n_{k_0}}(x_0))$. Observe that the segment $S''$ is $\varepsilon$-close to a segment of the same slope and same length as $S$.

Take a simple path $\tilde{\gamma}$ contained in $D$ whose extremities are the points $x$ and $y$ and that does not meet any lift of $\alpha$. Fix an essential simple closed curve $\beta$ of $\T ^2$ homotopic to $\alpha$ that is disjoint from $\alpha$ and has a lift to $\mathbb{R}^2$ containing the path $\tilde{\gamma}$. Observe that the path $p_{m}(\tilde{f}^{n_{k_0}}(\tilde{\gamma}))$ is homotopic with fixed extremities to $S''$ in $\T ^2_{m}$ (because those paths admit  lifts to $\mathbb{R}^2$ with the same endpoints). 
By Lemma \ref{LemTopology}, there exists a point $p$ of $\T^{2}_{m}$, which neither belongs to any lift of $\alpha$ nor to any lift of $\beta$, such that those two paths are still homotopic with fixed extremities in $\T^{2}_{m} \setminus \left\{ p \right\}$.
By Lemma \ref{LemGoodCover}, there exists a covering map $\Sigma(K) \rightarrow \T ^{2}_{m}$ which is ramified only at the point $p$. By Corollary~\ref{CorIrrationalLines}, any lift of the curve $p_{m}(\tilde{f}^{n_{k_0}}(\tilde{\beta}))$ to the surface $\Sigma (K)$ meets any horizontal curve of $\Sigma(K)$. However, the curve $p_{m}(\tilde{f}^{n_{k_0}}(\tilde{\beta}))$ is a lift of the curve $f^{n_{k_0}}(\beta)$ to $\T^{2}_{m}$ so that, by Lemma \ref{LemLift},
$$ d_{C^{\dagger}(\T^{2}_{m})}\big(p_{m}(\tilde{f}^{n_{k_0}}(\tilde{\beta})),p_{m}(\tilde{\alpha})\big) \leq K.$$
Hence the curves $p_{m}(\tilde{f}^{n_{k_0}}(\tilde{\beta}))$ and $p_{m}(\tilde{\alpha})$ must admit disjoint lifts to $\Sigma(K)$, which is a contradiction as the latter curve is horizontal.
\end{proof}

\section{Possible directions of good limit values}\label{SecPossible}

The following proposition gives two possibilities for the possible shapes of the good limit values of the sequence $(A_n)_{n}$.

\begin{proposition}\label{PropDichot}
Let $f \in \mathrm{Homeo}_0(\T ^2)$.
One of the following holds.
\begin{enumerate}
\item There exists a good limit value of the sequence $(A_n)_{n }$ which contains a nontrivial segment with irrational slope.
\item There exists a line with rational slope which contains any good limit value of the sequence $(A_n)_{n}$.
\end{enumerate}
\end{proposition}

\begin{proof}
Suppose that 1.~does not hold. Then any good limit value is a convex set with empty interior (any convex set with nonempty interior contains a nontrivial segment with irrational slope). Hence, any good limit value is a segment of rational slope, contained in the closed unit disk, and having diameter 1. 

We argue by contradiction by supposing that there are at least two different rational directions containing a good limit value.
Call these directions $\theta_1,\theta_2\in\Pj(\R^2)$. 

We endow $\Pj(\R^2)$ with a distance $\delta$ making it homeomorphic to the circle.
For $\theta\in\Pj(\R^2)$, denote $L_\theta$ the line of direction $\theta$ passing by 0.

As a first step, we state that if $d_n$ is large enough, then $A_n$ is very close to a rational segment. This follows from a simple compactness argument.

\begin{claim}
For any $\varep>0$, there exists $R>0$ such that if $d_n>R$, then there exists $\theta\in\Pj(\R^2)$ rational such that the following holds:
\begin{equation}\tag{$P_{n,\theta,\varep}$}\label{PropP}
\forall x\in A_n,\, d(x,L_\theta)\le\varep.
\end{equation}
\end{claim}

\begin{proof}
Suppose the claim is false. Then there exists $\varep_0>0$ and a subsequence $(n_k)$ such that $d_{n_k}\to+\infty$, and that for any rational direction $\theta\in\Pj(\R^2)$, there exists $x\in A_{n_k}$ with $d(x,L_\theta) > \varepsilon_0$.

By taking a subsequence of $(n_k)_k$ if necessary, one can suppose that $d_{n_k}\ge k$ and that the sets $A_{n_k}$ converge towards some compact set for Hausdorff topology. By the discussion at the beginning of the proof of the proposition, this set can only be a segment with rational slope $\theta_0$, \emph{i.e.} contained in some line $L_{\theta_0}$. Hence, if $k$ is large enough, for any $x\in A_{n_k}$ we have $d(x,L_{\theta_0})\le\varep_0$. This is a contradiction.
\end{proof}

It allows to define the following.

\begin{definition}
A \emph{main direction} of $A_n$ for $\varep>0$ is a direction $\theta$ such that $(P_{n,\theta,\varep})$ holds. 

An \emph{$R$-excursion} is any integer interval $[n_1,n_2]$ on which any $n \in [n_1,n_2]$ satisfies $d_n \ge R$.
\end{definition}

The following says that if $R$ is large enough, then on any $R$-excursion, the main directions of $A_n$ cannot vary a lot. 

\begin{lemma}\label{LemExcursion}
For any $\delta_0>0$, there exists $\varep>0$ and $R>0$ satisfying: if $[n_1,n_2]$ is an $R$-excursion, then for any $n,n'\in [n_1,n_2]$ and any $\theta,\theta'$ such that $(P_{n,\theta,\varep})$ and $(P_{n',\theta',\varep})$ hold, we have $\delta(\theta,\theta')<\delta_0$.
\end{lemma}

Note that if the conclusion of the lemma holds for some $\varep>0$, then it holds for any $0<\varep'<\varep$.

\begin{proof}
Suppose it is false. Then there exists $\delta_0>0$ such that for any $k>0$, there is a $k$-excursion $[n_1^k,n_2^k]$ and two directions $\alpha_1^k,\alpha_2^k$ with $\delta(\alpha_1^k,\alpha_2^k)\ge\delta_0$ such that $(P_{n_1^k,\alpha_1^k,1/(2k)})$ and $(P_{n_2^k,\alpha_2^k,1/(2k)})$ hold. By extracting a subsequence, one can suppose that $(\alpha_1^k,\alpha_2^k)$ converge towards $(\alpha_1,\alpha_2)$, and more precisely that $\delta(\alpha_1^k,\alpha_1)\le 1/(2k)$ and $\delta(\alpha_2^k,\alpha_2)\le 1/(2k)$. In this case, for $k$ large enough, $(P_{n_1^k,\alpha_1,1/k})$ and $(P_{n_2^k,\alpha_2,1/k})$ hold.
 
We now prove that for $n\in [n_1^k,n_2^k]$, the set of main directions cannot vary a lot between times $n$ and $n+1$. Indeed, calling $K_f = 2\max\big(d(\tilde f,\Id_{\R^2}), d(\tilde f^{-1},\Id_{\R^2})\big)$, we have
\[d_H\big(\tilde f^n(D)-\tilde f^n(\tilde x_0),\tilde f^{n+1}(D) - \tilde f^{n+1}(\tilde x_0)\big) \le K_f\]
Hence,

\begin{equation}\label{EqMajorKf1}
d_H\left(\frac{\tilde f^n(D)-\tilde f^n(\tilde x_0)}{d_n},\frac{\tilde f^{n+1}(D) - \tilde f^{n+1}(\tilde x_0)}{d_n}\right) \le \frac{K_f}{d_n}.
\end{equation}
But we also have $|d_n-d_{n+1}| \le K_f$, so $|1-d_{n+1}/d_n| \le K_f/d_n$. The fact that --- given a compact subset $A$ of the unit disc --- the map $\R_+\ni\lambda \mapsto \lambda A$ is $1$-Lipschitz for Hausdorff distance $d_H$ implies that 
\begin{align*}
d_H\bigg(\frac{\tilde f^{n+1}(D) -\tilde f^{n+1}(\tilde x_0)}{d_n} & ,\frac{\tilde f^{n+1}(D) - \tilde f^{n+1}(\tilde x_0)}{d_{n+1}}\bigg)\\
& = d_H\left(\frac{d_{n+1}}{d_n}\frac{\tilde f^{n+1}(D)-\tilde f^{n+1}(\tilde x_0)}{d_{n+1}},\frac{\tilde f^{n+1}(D) - \tilde f^{n+1}(\tilde x_0)}{d_{n+1}}\right)\\
& \le \left|1-\frac{d_{n+1}}{d_n}\right| \le \frac{K_f}{d_n}.
\end{align*}
Combined with \eqref{EqMajorKf1}, by triangle inequality, this gives
\[d_H(A_n,A_{n+1}) = d_H\left(\frac{\tilde f^n(D)-\tilde f^n(\tilde x_0)}{d_n},\frac{\tilde f^{n+1}(D) - \tilde f^{n+1}(\tilde x_0)}{d_{n+1}}\right) \le \frac{2K_f}{d_n}.\]

Using the fact that $d_n$ tends to infinity, we deduce that the Hausdorff distance between $A_n$ and $A_{n+1}$ is in $O(1/k)$ (recall that $n\in [n_1^k,n_2^k]$).
Hence, if $(\theta_n)$ and $(\theta'_n)$ are such that $(P_{n,\theta_n,1/k})$ and $(P_{n+1,\theta'_n,1/k})$ hold, then $\delta(\theta_n,\theta'_n)$ tends to 0 as $k$ tends to infinity. This implies for any $\eta>0$, the directions $\alpha_1$ and $\alpha_2$ can be joined by an $\eta$-chain consisting of accumulation points of sequences $(\theta_{n_k})$, where $n_k \in [n_1^k,n_2^k]$, and $(P_{n_k,\theta_{n_k},1/k})$ holds.

This proves that the set of $\theta\in\Pj(\R^2)$ satisfying the following property contains an interval containing both $\alpha_1$ and $\alpha_2$: there exists a sequence $n_k \in[n_1^k,n_2^k]$ such that $(P_{n_k,\theta_{n_k},1/k})$ holds, $d_{n_k} \geq k$ and $\theta_{n_k} \to \theta$.
This is a contradiction as such an interval has to contain an irrational direction. 
\end{proof}

We are now ready to end the proof of Proposition~\ref{PropDichot}. 
The idea is that the images of the fundamental domain cannot grow in two different directions on two $R$-excursions of the same length. We first build such two $R$-excursions, and then use them to get to a contradiction.

Let $\delta_0<\delta(\theta_1,\theta_2)/3$, and consider 
\[\alpha = \min\{\angle(\theta,\theta')\mid \delta(\theta_1,\theta)<\delta_0, \delta(\theta_2,\theta')<\delta_0\}.\]
Take $\varep>0$ such that $\sin\alpha\ge 6\varep$.

Consider $R$ associated to $\varep$ and $\delta_0$ as in Lemma~\ref{LemExcursion}. Increasing $R$ if necessary, one can suppose that $R\ge K_f$, with $K_f = 2\max\big(d(\tilde f,\Id_{\R^2}), d(\tilde f^{-1},\Id_{\R^2})\big)$.

We now consider two $R$-excursions of the same length, one with main directions close to $\theta_1$, the other with main directions close to $\theta_2$; we moreover suppose that on the first one, the diameter increases a lot. More precisely, we first prove that there exists $n_1\le n_2$ such that: 
\begin{itemize}
\item $[n_1,n_2]$ is an $R$-excursion;
\item $d_{n_2}\ge 4R/\varep$;
\item $d_{n_1}\le 2R$;
\item for any $n\in [n_1,n_2]$, there exists $\theta\in\Pj(\R^2)$ such that $\delta(\theta_1,\theta)<\delta_0$ and $(P_{n,\theta,\varep})$ holds.
\end{itemize} 
Indeed, by Lemma~\ref{LemExcursion}, for any $R'>R$, as by hypothesis there is at least two different rational directions containing a good limit value, there exists an infinite number of $R'$-excursions. Moreover, still by Lemma~\ref{LemExcursion}, there exists an infinite number of $R'$-excursions $[n_1,n_2]$ such that for any $n\in [n_1,n_2]$ and any $\theta$ such that $(P_{n,\theta,\varep})$ holds, we have $\delta(\theta,\theta_1)<\delta_0$.

So we can consider such a $4R/\varep$-excursion $[\tilde n_1,n_2]$, and $n_1$ the minimal integer such that $[n_1,n_2]$ is a $R$-excursion. Trivially, $[n_1,n_2]$ satisfies the two firsts points and the last point. But, as we have already seen, $|d_{n_1}-d_{n_1-1}| \le K_f$, so $d_{n_1}\le d_{n_1-1} + K_f\le 2R$.

As, for any $n$, $|d_{n+1}-d_n| \leq K_f$, observe that there are arbitrarily long $R$-excursions with one element $n$ of the excursion satisfying $(P_{n,\theta_2,\varep})$. As a consequence, using Lemma~\ref{LemExcursion}, one can similarly find $n'_1\le n'_2$ such that:
\begin{itemize}
\item $[n'_1,n'_2]$ is an $R$-excursion;
\item $n'_2-n'_1 = n_2-n_1$;
\item $d_{n'_1}\le 2R$;
\item for any $n\in [n'_1,n'_2]$, there exists $\theta\in\Pj(\R^2)$ such that $\delta(\theta_2,\theta)<\delta_0$ and $(P_{n,\theta,\varep})$ holds.
\end{itemize} 
Indeed, consider an $R+(n_2-n_1)K_f$-excursion $[\tilde n'_1, \tilde n'_2]$, consider $n'_1$ the minimal integer such that $[n'_1,\tilde n'_2]$ is an $R$-excursion, and set $n'_2 = n'_1+(n_2-n_1)$; the hypothesis on the size of the excursion ensures that $[n'_1,n'_2]$ is a $R$-excursion.
\bigskip

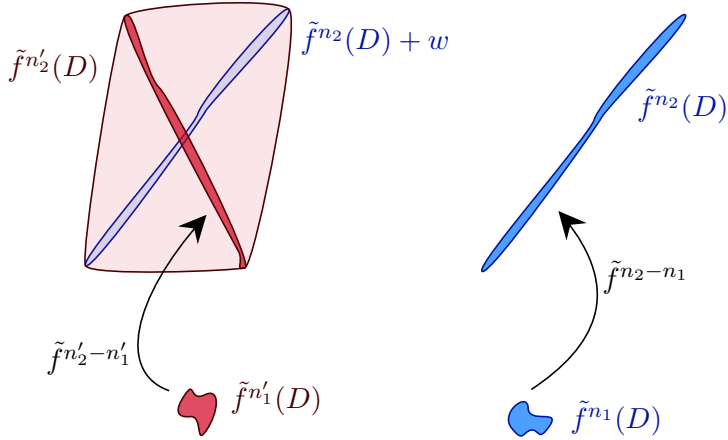
\begin{figure}
\begin{center}
\tikzset{every picture/.style={line width=0.6pt}} 

\begin{tikzpicture}[x=0.75pt,y=0.75pt,yscale=-1.3,xscale=1.3]

\draw  [color={rgb, 255:red, 0; green, 20; blue, 157 }  ,draw opacity=1 ][fill={rgb, 255:red, 0; green, 118; blue, 255 }  ,fill opacity=0.15 ] (191.65,91.68) .. controls (193.21,85.65) and (224.13,48.73) .. (226.58,51.21) .. controls (229.03,53.68) and (196.52,86.9) .. (194.46,90.9) .. controls (192.4,94.89) and (151.33,152.96) .. (148.58,150.46) .. controls (145.83,147.96) and (190.09,97.71) .. (191.65,91.68) -- cycle ;
\draw  [color={rgb, 255:red, 70; green, 0; blue, 0 }  ,draw opacity=1 ][fill={rgb, 255:red, 208; green, 0; blue, 30 }  ,fill opacity=.7 ] (186.69,200.38) .. controls (190.78,194.38) and (188.23,203.37) .. (194.38,200.38) .. controls (200.54,197.4) and (199.6,207.28) .. (195,214.08) .. controls (190.4,220.88) and (193.95,208.32) .. (188.69,208.85) .. controls (183.43,209.37) and (182.6,206.38) .. (186.69,200.38) -- cycle ;
\draw  [color={rgb, 255:red, 0; green, 20; blue, 157 }  ,draw opacity=1 ][fill={rgb, 255:red, 0; green, 118; blue, 255 }  ,fill opacity=.7 ] (312.06,213.26) .. controls (307.36,207.26) and (318.28,198.95) .. (318.42,205.46) .. controls (318.55,211.97) and (331.25,205.64) .. (326.65,212.44) .. controls (322.05,219.24) and (322.77,213.38) .. (319.48,214.08) .. controls (316.18,214.79) and (316.77,219.26) .. (312.06,213.26) -- cycle ;
\draw  [color={rgb, 255:red, 70; green, 0; blue, 0 }  ,draw opacity=1 ][fill={rgb, 255:red, 208; green, 0; blue, 30 }  ,fill opacity=.7 ] (163.67,53.94) .. controls (167.22,49.72) and (171.44,76.39) .. (176.78,81.5) .. controls (182.11,86.61) and (213.67,150.17) .. (209.67,151.28) .. controls (205.67,152.39) and (209.47,149.07) .. (207.22,144.83) .. controls (204.97,140.6) and (160.11,58.17) .. (163.67,53.94) -- cycle ;
\draw  [color={rgb, 255:red, 0; green, 20; blue, 157 }  ,draw opacity=1 ][fill={rgb, 255:red, 0; green, 118; blue, 255 }  ,fill opacity=.7 ] (343.82,93.84) .. controls (345.38,87.81) and (376.3,50.9) .. (378.75,53.38) .. controls (381.2,55.85) and (348.68,89.07) .. (346.63,93.06) .. controls (344.57,97.06) and (303.5,155.13) .. (300.75,152.63) .. controls (298,150.13) and (342.26,99.87) .. (343.82,93.84) -- cycle ;
\draw  [color={rgb, 255:red, 70; green, 0; blue, 0 }  ,draw opacity=1 ][fill={rgb, 255:red, 208; green, 0; blue, 30 }  ,fill opacity=0.1 ] (163.67,53.94) .. controls (167.22,49.72) and (221.25,46.1) .. (226.58,51.21) .. controls (231.92,56.32) and (213.67,150.17) .. (209.67,151.28) .. controls (205.67,152.39) and (150.83,154.69) .. (148.58,150.46) .. controls (146.34,146.22) and (160.11,58.17) .. (163.67,53.94) -- cycle ;
\draw    (180.67,198.83) .. controls (160.74,191.81) and (166.1,163.58) .. (193.31,132.65) ;
\draw [shift={(195,130.75)}, rotate = 132.15] [fill={rgb, 255:red, 0; green, 0; blue, 0 }  ][line width=0.08]  [draw opacity=0] (10.72,-5.15) -- (0,0) -- (10.72,5.15) -- (7.12,0) -- cycle    ;
\draw    (319.67,198.33) .. controls (343.88,182.16) and (355.89,160.85) .. (331.92,131.52) ;
\draw [shift={(330,129.25)}, rotate = 49.01] [fill={rgb, 255:red, 0; green, 0; blue, 0 }  ][line width=0.08]  [draw opacity=0] (10.72,-5.15) -- (0,0) -- (10.72,5.15) -- (7.12,0) -- cycle    ;

\draw (203.17,201.8) node [anchor=west] [inner sep=0.75pt]  [color={rgb, 255:red, 70; green, 0; blue, 0 }  ,opacity=1 ]  {$\tilde{f}^{n'_{1}}( D)$};
\draw (333.63,209.38) node [anchor=west] [inner sep=0.75pt]  [color={rgb, 255:red, 0; green, 20; blue, 157 }  ,opacity=1 ]  {$\tilde{f}^{n_{1}}( D)$};
\draw (167.62,186.3) node [anchor=east] [inner sep=0.75pt]  [color={rgb, 255:red, 0; green, 0; blue, 0 }  ,opacity=1 ]  {$\tilde{f}^{n'_{2} -n'_{1}}$};
\draw (345.92,154.8) node [anchor=west] [inner sep=0.75pt]  [color={rgb, 255:red, 0; green, 0; blue, 0 }  ,opacity=1 ]  {$\tilde{f}^{n_{2} -n_{1}}$};
\draw (154.87,73.3) node [anchor=east] [inner sep=0.75pt]  [color={rgb, 255:red, 70; green, 0; blue, 0 }  ,opacity=1 ]  {$\tilde{f}^{n'_{2}}( D)$};
\draw (230,55) node [anchor=north west][inner sep=0.75pt]  [color={rgb, 255:red, 0; green, 20; blue, 157 }  ,opacity=1 ]  {$\tilde{f}^{n_{2}}( D)+w$};
\draw (359.13,80.98) node [anchor=north west][inner sep=0.75pt]  [color={rgb, 255:red, 0; green, 20; blue, 157 }  ,opacity=1 ]  {$\tilde{f}^{n_{2}}( D)$};

\end{tikzpicture}
\caption{Proof of Proposition~\ref{PropDichot}. If $\tilde f^{n_2}(D)$ and $\tilde f^{n'_2}(D)$ have more or less the same size but different main directions, then it forces $\tilde f^{n'_2}(D)$ to have interior (light red shape): the shape in deep red is impossible for $\tilde f^{n'_2}(D)$ (recall that $n_2-n_1 = n'_2-n'_1$) because of the shape of the integer translate $\tilde{f}^{n_{2}}( D)+w$ of $\tilde{f}^{n_{2}}( D)$. This forces the relation $d'_{n_2}\gg d_{n_2}$, but a symmetric argument implies that $d_{n_2}\gg d'_{n_2}$, leading to a contradiction.}\label{FigPropDichot}
\end{center}
\end{figure}

The contradiction then comes as follows (see Figure~\ref{FigPropDichot}).
Take $x_1\in D$ such that $\|\tilde f^{n_2}(x_1)-\tilde f^{n_2}(x_0)\|\ge d_{n_2}/2$.  Note that $\|\tilde f^{n_1}(x_1)-\tilde f^{n_1}(x_0)\|\le 2R$. 
 There exists $v_0,v_1\in\Z^2$ such that $\tilde f^{n_1}(x_0)-v_0$ and $\tilde f^{n_1}(x_1)-v_1$ both belong to the fundamental domain $\tilde f^{n'_1}(D)$. Hence,
\begin{align*}
\|v_0-v_1\|
& \le \big\|\big(v_0-\tilde f^{n_1}(x_0)\big) - \big(v_1-\tilde f^{n_1}(x_1)\big)\big\|+\|\tilde f^{n_1}(x_0)-\tilde f^{n_1}(x_1)\| \\
& \le d_{n'_1}+d_{n_1} \le 4R.
\end{align*}
Now, take any $\theta$ such that $\delta(\theta,\theta_2)\le \delta_0$, and consider a unit vector $u_\theta$ orthogonal to $L_\theta$.
\begin{align*}
\operatorname{dist} :&= d\Big(\tilde f^{n_2-n_1}\big(\tilde f^{n_1}(x_0)-v_0\big) -\tilde f^{n_2-n_1}\big(\tilde f^{n_1}(x_1)-v_1\big),L_{\theta}\Big)\\
& = \Big|\langle \tilde f^{n_2}(x_0)-\tilde f^{n_2}(x_1),u_\theta\rangle + \langle v_0-v_1,u_\theta\rangle\Big|\\
& \ge \big\|\tilde f^{n_2}(x_0)-\tilde f^{n_2}(x_1)\big\|\sin\Big(\angle \big(\tilde f^{n_2}(x_0)-\tilde f^{n_2}(x_1),\theta\big)\Big) - 4R \\
& \ge \frac{d_{n_2}}{2}\sin\alpha - \varep d_{n_2}\qquad \text{(this is here we use $d_{n_2}\ge 4R/\varep$)}\\
& \ge \frac{d_{n_2}}{2}6\varep - \varep d_{n_2}\\
& \ge 2\varep d_{n_2}.
\end{align*} 
As the two points $\tilde f^{n_1}(x_0)-v_0$ and $\tilde f^{n_1}(x_1)-v_1$ belong to $\tilde f^{n'_1}(D)$, and as $(P_{{n'_2},\theta,\varep})$ holds for some $\theta$ such that $\delta(\theta,\theta_2)\le \delta_0$, then $\operatorname{dist}$ has to be smaller than $\varep d_{n'_2}$ for some $\theta$ such that $\delta(\theta,\theta_2)\le \delta_0$.
Hence, $d_{n'_2}\ge 2 d_{n_2}$. But then one can apply the exact same argument, permuting $\theta_1$ with $\theta_2$, to deduce that $d_{n_2}\ge 2 d_{n'_2}$.
This is a contradiction. 
\end{proof}

\section{Rational case and end of the proof of Theorem \ref{ThmElliptictorus}}\label{SecFinal}

We now state the last result we need to prove Theorem \ref{ThmElliptictorus}.

\begin{proposition}\label{LastProp}
Let $f \in \mathrm{Homeo}(\T ^2)$.
Suppose that any good limit value of the sequence $(A_n)$ is contained in the horizontal axis. Then either $f$ has bounded deviation in the vertical direction, or there exists a sequence $(a_n)$ of positive real numbers tending to infinity, and a sequence $(w_n)$ of vectors of $\R^2$, such that some limit value of the sequence 
\[(B_n+w_n)_n = \left(\frac{1}{a_n}\Big(\tilde f^n(D)-\tilde f^n(x_0)\Big) + w_n\right)_n\]
contains $B(0,1)$.
\end{proposition}

Let us first show how this proposition implies Theorem \ref{ThmElliptictorus}.

\begin{proof}[Proof of Theorem \ref{ThmElliptictorus}]
Let $f \in \mathrm{Homeo}_0(\T ^2)$, and suppose that in any rational direction, $f$ has no bounded deviation. We want to prove that $f$ does not act elliptically.

Apply Proposition \ref{PropDichot}. In the first case given by this proposition, Proposition~\ref{PropParabolicityCriterion} implies that $f$ does not act elliptically on $C^{\dagger}(\T ^2)$.

Suppose now that the second case given by Proposition~\ref{PropDichot} holds: There exists a line with rational slope which contains any good limit value of the sequence $(A_n)_{n}$. By conjugating with an element of $SL_2(\Z)$ if necessary, we do not lose generality by supposing that this direction is horizontal. 

Apply Proposition~\ref{LastProp}. As $f$ has unbounded deviation in the vertical direction, there exists a sequence $(a_k)$ of positive real numbers tending to infinity, and a sequence $(w_k)$ of vectors of $\R^2$, such that some limit value of the sequence $(B_n+w_k)_k$ contains $B(0,1)$.

The parabolicity criterion (Proposition~\ref{PropParabolicityCriterion}) applies and shows that $f$ does not act elliptically on $C^{\dagger}(\T ^2)$.
\bigskip

Now, suppose that $f \in \mathrm{Homeo}(\T ^2)$ has an iterate homotopic to a linear Anosov homeomorphism $A$. Then $f$ has unbounded deviation from any direction, and by \cite[Theorem 5.3]{BHMMW}, $f$ acts hyperbolically on $C^\dagger(\T^2)$. Note that we could also use our ellipticity criterion (Proposition~\ref{PropParabolicityCriterion}) to conclude that $f$ does not act elliptically: consider $(0,0), (1,1)\in D$, then $\tilde f^n (1,1)-\tilde f^n(0,0) = A^n (1,1)$; using the fact that these vectors tend to some irrational direction of $\Pj(\R^2)$ together with the quasi-convexity of fundamental domains we conclude that some limit set $B_\infty$ contains a nontrivial segment of irrational slope.
\bigskip

Finally, suppose that $f \in \mathrm{Homeo}(\T^2)$ has an iterate homotopic to a Dehn twist. Conjugating by an element of $SL_2(\mathbb{Z})$ if necessary, we can suppose that $f^n$ is homotopic to a Dehn twist  about the horizontal direction. Suppose that $f$ has unbounded displacement in the vertical direction. Then there is some $k\in\Z^*$ such that $\tilde f^n (0,1)-\tilde f^n(0,0) = (nk,1)$. Hence any good limit value contains a nontrivial horizontal interval. If there is a good limit value which is not included in the horizontal axis, as such a limit value is convex and contains a nontrivial horizontal interval, then Proposition~\ref{PropParabolicityCriterion} applies and $f$ does not act elliptically on $C^\dagger(\T^2)$. If not, then any good limit value is included in the horizontal axis, and Proposition~\ref{LastProp} allows to once again apply Proposition~\ref{PropParabolicityCriterion} to prove that $f$ does not act elliptically on $C^\dagger(\T^2)$.
\end{proof}

\begin{proof}[Proof of Proposition~\ref{LastProp}]
In this proof, for a set $A$ and $r\ge 0$, we denote $B(A,r) = \{x\mid d(x,A)\le r\}$.

Suppose that any good limit value of the sequence $(A_n)_n$ is included in the horizontal axis, and that $f$ has unbounded displacement in the vertical direction.

Applying the same idea as in the proof of Theorem \ref{ThmElliptictorus} (3 paragraphs above), by considering the iterates of $(0,0)$ and $(1,1)$, we can see that under these conditions, an iterate of $f$ cannot be isotopic to a linear Anosov automorphism, or a Dehn twist in a direction that is not horizontal (adapting the points $(0,0)$ and $(1,1)$ in the latter case if necessary). So some iterate of $f$ is isotopic to the identity, or to a Dehn twist. Hence, replacing $f$ with an iterate of it if necessary, we can suppose that $f$ is homotopic to
\[\begin{pmatrix}
1 & k_0 \\ 0 & 1
\end{pmatrix}\]
for some $k_0\in\Z$.

\begin{claim}\label{LastClaim}
Suppose that any good limit value of the sequence $(A_n)_n$ is contained in the horizontal axis. 
Suppose also that for any sequence $(a_n)$ of positive real numbers tending to infinity, and any sequence $(w_n)$ of vectors of $\R^2$, any limit value of the sequence $(B_n+w_n)_n$ does not contain $B(0,1)$.

Then there exists $C>0$ and, for any $n\ge 0$, a line $L_n$ passing through $0$ and of direction $\theta_n$ such that:
\begin{itemize}
\item $\theta_n$ tends to the horizontal direction $(1,0)$;
\item $\tilde f^n(D)-\tilde f^n(x_0)\subset B(L_n, C)$.
\end{itemize}
If moreover $f$ has unbounded deviation in the vertical direction then, up to taking a subsequence, we can moreover suppose the following:
\begin{itemize}
\item the projection of $\tilde f^n(D)$ on the vertical axis has length $h_n$ tending to infinity.
\end{itemize}
\end{claim}

\begin{proof}
For any $n\in\N$, let $x_n,y_n\in D$ such that $d(\tilde f^n(x_n),\tilde f^n(y_n)) = d_n$. Let $L'_n$ be the line passing by $\tilde f^n(x_n)$ and $\tilde f^n(y_n)$.

Let $b_n = \max\{d(\tilde f^n(z),L'_n)\mid z\in D\}$ and $z_n\in D$ be such that $d(\tilde f^n(z_n),L'_n) = b_n$. If $\sup_n b_n = C/2 < +\infty$, the two first points of the claim are proved, by setting $L_n = L'_n-\tilde f^n(x_0)$: in this case the distance of any point of $\tilde f^n(D)-\tilde f^n(x_0)$ to $L_n$ is smaller than $2C/2 = C$. 

Otherwise, there exists a subsequence $(n_k)$ along which $b_{n_k}\ge k$. Let $q_k$ be the orthogonal projection of $\tilde f^n(z_{n_k})$ on $L'_{n_k}$, $a_k = \sqrt{b_{n_k}}$. Note that $q_k\in [x_{n_k},y_{n_k}]$ because of the definition of $x_n$ and $y_n$ (for $a,b\in\R^2$, we denote by $[a,b]$ the affine segment between points $a$ and $b$).
Let also
\[w_k = \frac{q_k+\tilde f^{n_k}(x_0)}{a_k}.\]
Then the set 
\[B_{n_k}+w_k = \frac{\tilde f^{n_k}(D) - \tilde f^{n_k}(x_0)}{a_k}+w_k = \frac{\tilde f^{n_k}(D) - q_k}{a_k}\]
contains both the two points $(\tilde f^n(x_{n_k})-q_k)/a_k$ and $(\tilde f^n(y_{n_k})-q_k)/a_k$, which belong to the line $L_{n_k} = L'_{n_k}-q_k$ passing through 0, which are at distance $d_{n_k}/a_k \ge \sqrt k$, and such that the segment between these points contains 0; this set also contains the point $(\tilde f^n(z_{n_k})-q_k)/a_k$ that is at distance $\ge\sqrt k$ of the line $L_{n_k} = L'_{n_k}-q_k$.

By Lemma~\ref{LemConvexlimit}, we deduce that any limit value of the sequence $B_{n_k}+w_k$ contains a quarter of disk (centered at 0) of radius $10$. By modifying a bit the sequence $(w_k)$ to $(w'_k)$, \emph{i.e.} by applying a translation to each set $B_{n_k}+w_k$, we get a limit value of the sequence $(B_n+w'_n)_n$ containing $B(0,1)$.
\bigskip

For the last point of the claim, define
\[e_n = \sup\Big\{\big|\big\langle\tilde f^n(x)-\tilde f^n(y),\,(0,1)\big\rangle\big|\ \mid x,y\in D^2\Big\}\]
the ``diameter of $\tilde f^n(D)$ in the vertical direction''. By Lemma~\ref{LemDiamDeviation}, using that $f$ has unbounded deviation in the vertical direction, we have a subsequence $(n_k)$ along which $\lim e_{n_k} = +\infty$. This proves the last point.
\end{proof}

Now, up to increasing the constant $C$ of Claim~\ref{LastClaim} if necessary, suppose $C\ge 2$.
Let $m=5\lceil C\rceil$ and $n_1\in\N$ such that $h_{n_1}\ge 20m$ and $|\theta_{n_1}|\le 1/100$ (by Claim~\ref{LastClaim}). We denote by $p_1$ the projection on the first (horizontal) coordinate, and $p_2$ the projection on the second (vertical) one.

Let $n_2\ge n_1$ such that $|\theta_{n_2}| \ll |\theta_{n_1}|$, $|\tan \theta_{n_2}|d_{n_1}\le C$ and $d_{n_2}\gg d_{n_1}+k_0n_1C$. 

Let $x_2^-,x_2^+, x\in \tilde f^{n_2}(D)$ such that $p_1(x_2^-) = \min (p_1|_{\tilde{f}^{n_2}(D)})$, $p_1(x_2^+) = \max (p_1|_{\tilde{f}^{n_2}(D)})$, and $p_1(x) = \frac12(\min (p_1|_{\tilde{f}^{n_2}(D)})+\max (p_1|_{\tilde{f}^{n_2}(D)}))$ (see Figures~\ref{FigLastProp} and {Figxx1-}).

Let $v\in\Z^2$ such that $x\in \tilde f^{n_1}(D)+v$. We denote $D_0$ the integer translate of $D$ satisfying $\tilde f^{n_1}(D)+v = \tilde f^{n_1}(D_0)$. Let $x_1^-,x_1^+\in \tilde f^{n_1}(D_0)$ such that $d(x_1^-,x_1^+) = d_{n_1}$ and $p_1(x_1^-) < p_1(x_1^+)$. 

We suppose that
\begin{equation}\label{EqHypX}
p_{L_{n_1}}(x)\ge \frac12\big(p_{L_{n_1}}(x_1^-)+p_{L_{n_1}}(x_1^+)\big)
\end{equation}
(we identify $L_{n_1}$ with $\R$), the other inequality can be treated identically. Similarly, we suppose that $\theta_{n_1}>0$.

Let $n\in [m,2m]\cap \N$. Then
\[\tilde f^{n_1}\big(D_0+(0,n)\big) = \tilde f^{n_1}(D_0) + w_n,
\qquad\text{with}\quad
w_n = (k_0\,n_1\,n,n).\]

\begin{figure}
\begin{center}
\tikzset{every picture/.style={line width=0.75pt}} 

\begin{tikzpicture}[x=0.75pt,y=0.75pt,yscale=-1.3,xscale=1.3]

\draw [color={rgb, 255:red, 155; green, 155; blue, 155 }  ,draw opacity=1 ,dash pattern={on 4.5pt off 4.5pt}]  (160,170) -- (560,170) ;
\draw [color={rgb, 255:red, 0; green, 95; blue, 204 }  ,draw opacity=1 ]   (170,180) -- (550,160) ;
\draw [color={rgb, 255:red, 208; green, 2; blue, 27 }  ,draw opacity=1 ][fill={rgb, 255:red, 208; green, 2; blue, 27 }  ,fill opacity=1 ]   (295,215) -- (385,145) ;
\draw  [draw opacity=0][fill={rgb, 255:red, 208; green, 2; blue, 27 }  ,fill opacity=1 ] (382.36,145) .. controls (382.36,143.54) and (383.54,142.36) .. (385,142.36) .. controls (386.46,142.36) and (387.64,143.54) .. (387.64,145) .. controls (387.64,146.46) and (386.46,147.64) .. (385,147.64) .. controls (383.54,147.64) and (382.36,146.46) .. (382.36,145) -- cycle ;
\draw  [draw opacity=0][fill={rgb, 255:red, 208; green, 2; blue, 27 }  ,fill opacity=1 ] (292.36,215) .. controls (292.36,213.54) and (293.54,212.36) .. (295,212.36) .. controls (296.46,212.36) and (297.64,213.54) .. (297.64,215) .. controls (297.64,216.46) and (296.46,217.64) .. (295,217.64) .. controls (293.54,217.64) and (292.36,216.46) .. (292.36,215) -- cycle ;

\draw  [draw opacity=0][fill={rgb, 255:red, 0; green, 95; blue, 204 }  ,fill opacity=1 ] (167.36,180) .. controls (167.36,178.54) and (168.54,177.36) .. (170,177.36) .. controls (171.46,177.36) and (172.64,178.54) .. (172.64,180) .. controls (172.64,181.46) and (171.46,182.64) .. (170,182.64) .. controls (168.54,182.64) and (167.36,181.46) .. (167.36,180) -- cycle ;
\draw  [draw opacity=0][fill={rgb, 255:red, 0; green, 95; blue, 204 }  ,fill opacity=1 ] (547.36,160) .. controls (547.36,158.54) and (548.54,157.36) .. (550,157.36) .. controls (551.46,157.36) and (552.64,158.54) .. (552.64,160) .. controls (552.64,161.46) and (551.46,162.64) .. (550,162.64) .. controls (548.54,162.64) and (547.36,161.46) .. (547.36,160) -- cycle ;
\draw  [draw opacity=0][fill={rgb, 255:red, 0; green, 0; blue, 0 }  ,fill opacity=1 ] (349.79,170.29) .. controls (349.79,168.83) and (350.97,167.65) .. (352.43,167.65) .. controls (353.88,167.65) and (355.07,168.83) .. (355.07,170.29) .. controls (355.07,171.74) and (353.88,172.92) .. (352.43,172.92) .. controls (350.97,172.92) and (349.79,171.74) .. (349.79,170.29) -- cycle ;
\draw [color={rgb, 255:red, 208; green, 2; blue, 27 }  ,draw opacity=1 ][fill={rgb, 255:red, 208; green, 2; blue, 27 }  ,fill opacity=1 ]   (295.07,203.55) -- (385.07,133.55) ;
\draw  [draw opacity=0][fill={rgb, 255:red, 208; green, 2; blue, 27 }  ,fill opacity=1 ] (382.43,133.55) .. controls (382.43,132.09) and (383.62,130.91) .. (385.07,130.91) .. controls (386.53,130.91) and (387.71,132.09) .. (387.71,133.55) .. controls (387.71,135) and (386.53,136.18) .. (385.07,136.18) .. controls (383.62,136.18) and (382.43,135) .. (382.43,133.55) -- cycle ;
\draw  [draw opacity=0][fill={rgb, 255:red, 208; green, 2; blue, 27 }  ,fill opacity=1 ] (292.43,203.55) .. controls (292.43,202.09) and (293.62,200.91) .. (295.07,200.91) .. controls (296.53,200.91) and (297.71,202.09) .. (297.71,203.55) .. controls (297.71,205) and (296.53,206.18) .. (295.07,206.18) .. controls (293.62,206.18) and (292.43,205) .. (292.43,203.55) -- cycle ;

\draw [color={rgb, 255:red, 208; green, 2; blue, 27 }  ,draw opacity=1 ][fill={rgb, 255:red, 208; green, 2; blue, 27 }  ,fill opacity=1 ]   (295.09,192.77) -- (385.09,122.77) ;
\draw  [draw opacity=0][fill={rgb, 255:red, 208; green, 2; blue, 27 }  ,fill opacity=1 ] (382.45,122.77) .. controls (382.45,121.32) and (383.63,120.13) .. (385.09,120.13) .. controls (386.54,120.13) and (387.72,121.32) .. (387.72,122.77) .. controls (387.72,124.23) and (386.54,125.41) .. (385.09,125.41) .. controls (383.63,125.41) and (382.45,124.23) .. (382.45,122.77) -- cycle ;
\draw  [draw opacity=0][fill={rgb, 255:red, 208; green, 2; blue, 27 }  ,fill opacity=1 ] (292.45,192.77) .. controls (292.45,191.32) and (293.63,190.13) .. (295.09,190.13) .. controls (296.54,190.13) and (297.72,191.32) .. (297.72,192.77) .. controls (297.72,194.23) and (296.54,195.41) .. (295.09,195.41) .. controls (293.63,195.41) and (292.45,194.23) .. (292.45,192.77) -- cycle ;

\draw [color={rgb, 255:red, 208; green, 2; blue, 27 }  ,draw opacity=1 ][fill={rgb, 255:red, 208; green, 2; blue, 27 }  ,fill opacity=1 ]   (294.87,181.6) -- (384.87,111.6) ;
\draw  [draw opacity=0][fill={rgb, 255:red, 208; green, 2; blue, 27 }  ,fill opacity=1 ] (382.23,111.6) .. controls (382.23,110.14) and (383.41,108.96) .. (384.87,108.96) .. controls (386.32,108.96) and (387.5,110.14) .. (387.5,111.6) .. controls (387.5,113.05) and (386.32,114.23) .. (384.87,114.23) .. controls (383.41,114.23) and (382.23,113.05) .. (382.23,111.6) -- cycle ;
\draw  [draw opacity=0][fill={rgb, 255:red, 208; green, 2; blue, 27 }  ,fill opacity=1 ] (292.23,181.6) .. controls (292.23,180.14) and (293.41,178.96) .. (294.87,178.96) .. controls (296.32,178.96) and (297.5,180.14) .. (297.5,181.6) .. controls (297.5,183.05) and (296.32,184.23) .. (294.87,184.23) .. controls (293.41,184.23) and (292.23,183.05) .. (292.23,181.6) -- cycle ;

\draw  [dash pattern={on 4.5pt off 4.5pt}]  (325.14,153.67) .. controls (318.08,109.97) and (303.1,77.3) .. (247.54,72.51) ;
\draw [shift={(244.97,72.31)}, rotate = 3.97] [fill={rgb, 255:red, 0; green, 0; blue, 0 }  ][line width=0.08]  [draw opacity=0] (8.93,-4.29) -- (0,0) -- (8.93,4.29) -- (5.93,0) -- cycle    ;
\draw  [color={rgb, 255:red, 110; green, 0; blue, 15 }  ,draw opacity=1 ][fill={rgb, 255:red, 208; green, 2; blue, 27 }  ,fill opacity=1 ] (224.25,71.31) -- (232.61,71.31) -- (232.61,79.67) -- (224.25,79.67) -- cycle ;
\draw  [color={rgb, 255:red, 110; green, 0; blue, 15 }  ,draw opacity=1 ][fill={rgb, 255:red, 208; green, 2; blue, 27 }  ,fill opacity=1 ] (224.25,62.94) -- (232.61,62.94) -- (232.61,71.31) -- (224.25,71.31) -- cycle ;
\draw  [color={rgb, 255:red, 110; green, 0; blue, 15 }  ,draw opacity=1 ][fill={rgb, 255:red, 208; green, 2; blue, 27 }  ,fill opacity=1 ] (224.25,54.58) -- (232.61,54.58) -- (232.61,62.94) -- (224.25,62.94) -- cycle ;
\draw  [color={rgb, 255:red, 110; green, 0; blue, 15 }  ,draw opacity=1 ][fill={rgb, 255:red, 208; green, 2; blue, 27 }  ,fill opacity=1 ] (224.25,46.22) -- (232.61,46.22) -- (232.61,54.58) -- (224.25,54.58) -- cycle ;
\draw [color={rgb, 255:red, 0; green, 95; blue, 204 }  ,draw opacity=1 ]   (160.63,60.8) .. controls (204.63,46.8) and (198.74,77.87) .. (229.29,74.32) .. controls (259.84,70.78) and (194.38,53.78) .. (228.43,50.4) .. controls (262.48,47.03) and (258.08,75.87) .. (307.41,61.87) ;
\draw [color={rgb, 255:red, 0; green, 95; blue, 204 }  ,draw opacity=1 ] [dash pattern={on 0.84pt off 2.51pt}]  (144.19,64.58) .. controls (152.37,63.4) and (151.39,63.65) .. (160.63,60.8) ;
\draw [color={rgb, 255:red, 0; green, 95; blue, 204 }  ,draw opacity=1 ] [dash pattern={on 0.84pt off 2.51pt}]  (307.41,61.87) .. controls (315.33,59.49) and (316.67,55.57) .. (327.42,54.4) ;
\draw (351.79,173.69) node [anchor=north west][inner sep=0.75pt]    {$x$};
\draw (293,215) node [anchor=east] [inner sep=0.75pt]  [color={rgb, 255:red, 208; green, 2; blue, 27 }  ,opacity=1 ]  {$x_{1}^{-}$};
\draw (387,145) node [anchor=west] [inner sep=0.75pt]  [color={rgb, 255:red, 208; green, 2; blue, 27 }  ,opacity=1 ]  {$x_{1}^{+}$};
\draw (168,183.4) node [anchor=north east] [inner sep=0.75pt]  [color={rgb, 255:red, 0; green, 95; blue, 204 }  ,opacity=1 ]  {$x_{2}^{-}$};
\draw (552,156.6) node [anchor=south west] [inner sep=0.75pt]  [color={rgb, 255:red, 0; green, 95; blue, 204 }  ,opacity=1 ]  {$x_{2}^{+}$};
\draw (387.09,119.37) node [anchor=south west] [inner sep=0.75pt]  [font=\footnotesize,color={rgb, 255:red, 208; green, 2; blue, 27 }  ,opacity=1 ]  {$ \begin{array}{l}
\text{translates}\\
\text{of}\ \tilde{f}^{n_{1}}( D)
\end{array}$};
\draw (303.81,91.75) node [anchor=south west] [inner sep=0.75pt]    {$\tilde{f}^{-n_{1}}$};
\draw (233.85,180.52) node [anchor=north] [inner sep=0.75pt]  [font=\footnotesize,color={rgb, 255:red, 0; green, 95; blue, 204 }  ,opacity=1 ]  {$\ \tilde{f}^{n_{2}}( D)$};
\draw (170.52,67.19) node [anchor=north] [inner sep=0.75pt]  [font=\footnotesize,color={rgb, 255:red, 0; green, 95; blue, 204 }  ,opacity=1 ]  {$\ \tilde{f}^{n_{2} -n_{1}}( D)$};
\draw (209.75,117.37) node [anchor=south west] [inner sep=0.75pt]  [font=\footnotesize,color={rgb, 255:red, 208; green, 2; blue, 27 }  ,opacity=1 ]  {$ \begin{array}{l}
\text{translates}\\
\text{of}\ D
\end{array}$};
\end{tikzpicture}
\caption{End of proof of Proposition~\ref{LastProp}\label{FigLastProp} in the case of $k_0=0$. In the case $k_0\neq 0$ there is a shear appearing in the translates of $D$.}
\end{center}
\end{figure}
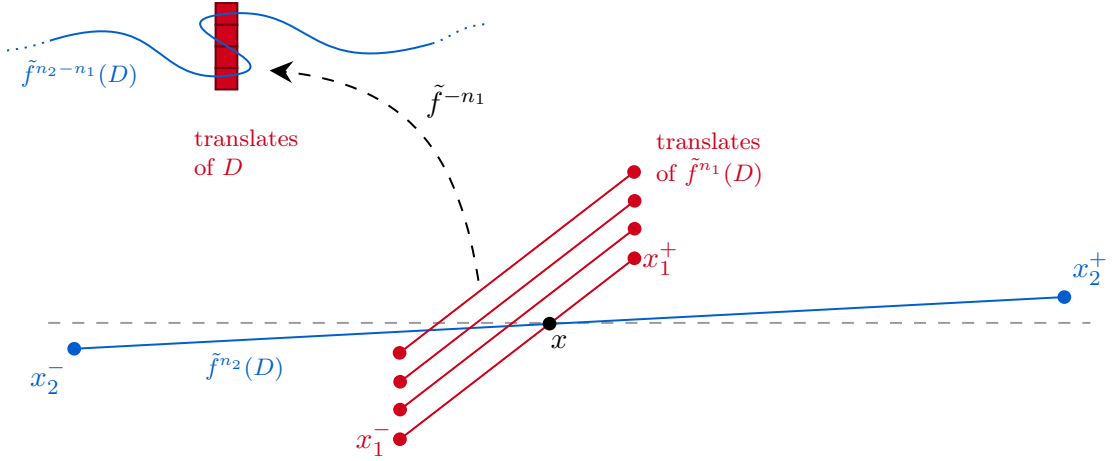

Let $\gamma_1$ be a path included in $\tilde f^{n_1}(D_0)$ linking $x_1^-$ to $x_1^+$, and $\gamma_2$ a path included in $\tilde f^{n_2}(D)$ linking $x_2^-$ to $x_2^+$. We want to prove that the paths $\gamma_1+w_n$ and $\gamma_2$ intersect for any $n\in [m,2m]$. We first define an affine shear mapping $A$ with linear part of the form 
\[\begin{pmatrix}
1 & 0 \\ -\tan\theta'_{n_2} & 1
\end{pmatrix}\]
for some angle $\theta'_{n_2}$, 
such that the abscissa of $Ax$ is 0, and that the points $Ax_2^-,Ax_2^+$ are on the horizontal axis $\{y=0\}$. Note that it forces the angle $\theta'_{n_2}$ to be close to $\theta_{n_2}$, in the sense that $\theta'_{n_2}-\theta_{n_2}\ll \theta_{n_2}$ (because of the bound by $C$, and the fact that $h_{n_2}$ goes to infinity).

Let us write $Ax_1^- = (a_1^-,b_1^-)$, $Ax_1^+ = (a_1^+,b_1^+)$, $Ax_2^- = (-M,0)$ and $Ax_2^+ = (M,0)$.
The fact that $d_{n_2}\gg d_{n_1}$ implies that 
\begin{equation}\label{EqA1-}
\max(|a_1^-|, |a_1^+|) \leq \max \big(d(x,x_1^-),d(x,x_1^+)\big)  \le d_{n_1}\le M/2.
\end{equation}
Hence, because $k_0n_1m \le 6k_0n_1C\ll d_{n_2}$,
\begin{equation}\label{EqA1-2}
\max\big(|a_1^-+k_0 n_1 n |, |a_1^++k_0 n_1 n |\big) \le M;
\end{equation}
note that $a_1^-+k_0 n_1 n$ is the abscissa of $A(x_1^-+w_n)$ and $a_1^++k_0 n_1 n$ is the abscissa of $A(x_1^++w_n)$.
The same estimates hold for any point of $\gamma_1+w_n$.

We also have (see Figure~\ref{Figxx1-} for the notations and the configuration.)
\begin{align*}
p_2(x_1^- -x) & \le p_2(P-Q)+2C \\
& = -d(P,Q)\sin\theta_{n_1}+2C\\
& \le -\frac{d_{n_1}-2C}{2}\sin\theta_{n_1}+2C.
\end{align*}
The last inequality comes from the fact that, by the hypothesis \eqref{EqHypX} on $x$, we have that $d(P,Q)\ge d(P,R)/2$. This implies that

\begin{figure}
\begin{center}
\tikzset{every picture/.style={line width=0.75pt}} 

\begin{tikzpicture}[x=0.75pt,y=0.75pt,yscale=-1,xscale=1]

\draw  [color={rgb, 255:red, 155; green, 155; blue, 155 }  ,draw opacity=1 ][fill={rgb, 255:red, 155; green, 155; blue, 155 }  ,fill opacity=0.1 ] (164.18,170) .. controls (164.18,155.74) and (175.74,144.18) .. (190,144.18) .. controls (204.26,144.18) and (215.82,155.74) .. (215.82,170) .. controls (215.82,184.26) and (204.26,195.82) .. (190,195.82) .. controls (175.74,195.82) and (164.18,184.26) .. (164.18,170) -- cycle ;
\draw [color={rgb, 255:red, 155; green, 155; blue, 155 }  ,draw opacity=1 ] [dash pattern={on 4.5pt off 4.5pt}]  (340,130) -- (450,130) ;
\draw [color={rgb, 255:red, 155; green, 155; blue, 155 }  ,draw opacity=1 ]   (409.08,111.94) .. controls (410.25,116.81) and (410.25,124.69) .. (410.08,130.06) ;
\draw  [color={rgb, 255:red, 155; green, 155; blue, 155 }  ,draw opacity=1 ][fill={rgb, 255:red, 155; green, 155; blue, 155 }  ,fill opacity=0.1 ] (314.18,130) .. controls (314.18,115.74) and (325.74,104.18) .. (340,104.18) .. controls (354.26,104.18) and (365.82,115.74) .. (365.82,130) .. controls (365.82,144.26) and (354.26,155.82) .. (340,155.82) .. controls (325.74,155.82) and (314.18,144.26) .. (314.18,130) -- cycle ;
\draw  [fill={rgb, 255:red, 0; green, 0; blue, 0 }  ,fill opacity=1 ] (332.33,111.03) .. controls (332.33,110.21) and (332.99,109.56) .. (333.81,109.56) .. controls (334.62,109.56) and (335.28,110.21) .. (335.28,111.03) .. controls (335.28,111.84) and (334.62,112.5) .. (333.81,112.5) .. controls (332.99,112.5) and (332.33,111.84) .. (332.33,111.03) -- cycle ;
\draw  [fill={rgb, 255:red, 0; green, 0; blue, 0 }  ,fill opacity=1 ] (194.81,188.92) .. controls (194.81,188.1) and (195.47,187.44) .. (196.29,187.44) .. controls (197.1,187.44) and (197.76,188.1) .. (197.76,188.92) .. controls (197.76,189.73) and (197.1,190.39) .. (196.29,190.39) .. controls (195.47,190.39) and (194.81,189.73) .. (194.81,188.92) -- cycle ;
\draw [color={rgb, 255:red, 155; green, 155; blue, 155 }  ,draw opacity=1 ]   (340,130) -- (333.21,152.04) ;
\draw [shift={(332.32,154.91)}, rotate = 287.13] [fill={rgb, 255:red, 155; green, 155; blue, 155 }  ,fill opacity=1 ][line width=0.08]  [draw opacity=0] (7.14,-3.43) -- (0,0) -- (7.14,3.43) -- (4.74,0) -- cycle    ;
\draw [color={rgb, 255:red, 155; green, 155; blue, 155 }  ,draw opacity=1 ]   (190,170) -- (181.52,148.8) ;
\draw [shift={(180.41,146.02)}, rotate = 68.2] [fill={rgb, 255:red, 155; green, 155; blue, 155 }  ,fill opacity=1 ][line width=0.08]  [draw opacity=0] (7.14,-3.43) -- (0,0) -- (7.14,3.43) -- (4.74,0) -- cycle    ;
\draw  [color={rgb, 255:red, 155; green, 155; blue, 155 }  ,draw opacity=1 ][fill={rgb, 255:red, 155; green, 155; blue, 155 }  ,fill opacity=0.1 ] (463.96,90.14) .. controls (463.96,75.88) and (475.52,64.32) .. (489.79,64.32) .. controls (504.05,64.32) and (515.61,75.88) .. (515.61,90.14) .. controls (515.61,104.4) and (504.05,115.97) .. (489.79,115.97) .. controls (475.52,115.97) and (463.96,104.4) .. (463.96,90.14) -- cycle ;
\draw  [fill={rgb, 255:red, 0; green, 0; blue, 0 }  ,fill opacity=1 ] (494.66,110.67) .. controls (494.66,109.85) and (495.32,109.19) .. (496.13,109.19) .. controls (496.95,109.19) and (497.6,109.85) .. (497.6,110.67) .. controls (497.6,111.48) and (496.95,112.14) .. (496.13,112.14) .. controls (495.32,112.14) and (494.66,111.48) .. (494.66,110.67) -- cycle ;
\draw [color={rgb, 255:red, 155; green, 155; blue, 155 }  ,draw opacity=1 ]   (489.79,90.14) -- (478.02,71) ;
\draw [shift={(476.45,68.45)}, rotate = 58.42] [fill={rgb, 255:red, 155; green, 155; blue, 155 }  ,fill opacity=1 ][line width=0.08]  [draw opacity=0] (7.14,-3.43) -- (0,0) -- (7.14,3.43) -- (4.74,0) -- cycle    ;
\draw  [dash pattern={on 0.84pt off 2.51pt}]  (190,170) -- (196.29,188.92) ;
\draw   (194.5,168.53) -- (196.04,173.26) -- (191.55,174.73) ;
\draw  [dash pattern={on 0.84pt off 2.51pt}]  (489.84,91.75) -- (496.13,110.67) ;
\draw   (491.35,94.73) -- (486.64,96.34) -- (485.07,91.75) ;
\draw  [color={rgb, 255:red, 65; green, 117; blue, 5 }  ,draw opacity=1 ][fill={rgb, 255:red, 65; green, 117; blue, 5 }  ,fill opacity=1 ] (188.53,170) .. controls (188.53,169.19) and (189.19,168.53) .. (190,168.53) .. controls (190.81,168.53) and (191.47,169.19) .. (191.47,170) .. controls (191.47,170.81) and (190.81,171.47) .. (190,171.47) .. controls (189.19,171.47) and (188.53,170.81) .. (188.53,170) -- cycle ;
\draw  [color={rgb, 255:red, 65; green, 117; blue, 5 }  ,draw opacity=1 ][fill={rgb, 255:red, 65; green, 117; blue, 5 }  ,fill opacity=1 ] (338.74,129.79) .. controls (338.74,128.97) and (339.4,128.31) .. (340.22,128.31) .. controls (341.03,128.31) and (341.69,128.97) .. (341.69,129.79) .. controls (341.69,130.6) and (341.03,131.26) .. (340.22,131.26) .. controls (339.4,131.26) and (338.74,130.6) .. (338.74,129.79) -- cycle ;
\draw  [dash pattern={on 0.84pt off 2.51pt}]  (333.81,111.03) -- (340.09,129.94) ;
\draw   (335.78,131.43) -- (334.05,126.76) -- (338.48,125.12) ;
\draw  [color={rgb, 255:red, 65; green, 117; blue, 5 }  ,draw opacity=1 ][fill={rgb, 255:red, 65; green, 117; blue, 5 }  ,fill opacity=1 ] (488.31,90.14) .. controls (488.31,89.33) and (488.97,88.67) .. (489.79,88.67) .. controls (490.6,88.67) and (491.26,89.33) .. (491.26,90.14) .. controls (491.26,90.96) and (490.6,91.62) .. (489.79,91.62) .. controls (488.97,91.62) and (488.31,90.96) .. (488.31,90.14) -- cycle ;
\draw [color={rgb, 255:red, 208; green, 2; blue, 27 }  ,draw opacity=1 ]   (190,170) -- (490,90) ;

\draw (412.59,119.12) node [anchor=west] [inner sep=0.75pt]  [color={rgb, 255:red, 128; green, 128; blue, 128 }  ,opacity=1 ]  {$\theta _{n_{1}}{}$};
\draw (332.32,112.38) node [anchor=north east] [inner sep=0.75pt]    {$x$};
\draw (198.46,188.15) node [anchor=north west][inner sep=0.75pt]    {$x_{1}^{-}$};
\draw (335.2,133.82) node [anchor=north east] [inner sep=0.75pt]  [color={rgb, 255:red, 128; green, 128; blue, 128 }  ,opacity=1 ]  {$C$};
\draw (188.02,154.25) node [anchor=west] [inner sep=0.75pt]  [color={rgb, 255:red, 128; green, 128; blue, 128 }  ,opacity=1 ]  {$C$};
\draw (497.13,113.32) node [anchor=north west][inner sep=0.75pt]    {$x_{1}^{+}$};
\draw (484.98,74.16) node [anchor=west] [inner sep=0.75pt]  [color={rgb, 255:red, 128; green, 128; blue, 128 }  ,opacity=1 ]  {$C$};
\draw (188.52,171.26) node [anchor=north east] [inner sep=0.75pt]  [color={rgb, 255:red, 65; green, 117; blue, 5 }  ,opacity=1 ]  {$P$};
\draw (341.74,131.85) node [anchor=north west][inner sep=0.75pt]  [color={rgb, 255:red, 65; green, 117; blue, 5 }  ,opacity=1 ]  {$Q$};
\draw (289.66,140.34) node [anchor=south east] [inner sep=0.75pt]  [color={rgb, 255:red, 208; green, 2; blue, 27 }  ,opacity=1 ]  {$L_{n_{1}}$};
\draw (492,90) node [anchor=west] [inner sep=0.75pt]  [color={rgb, 255:red, 65; green, 117; blue, 5 }  ,opacity=1 ]  {$R$};

\end{tikzpicture}

\caption{End of proof of Proposition~\ref{LastProp}: estimation of $p_2(x_1^--x)$. The length $d(P,R)$ of the red segment is bigger than $d_{n_1}-2C$.}\label{Figxx1-}
\end{center}
\end{figure}
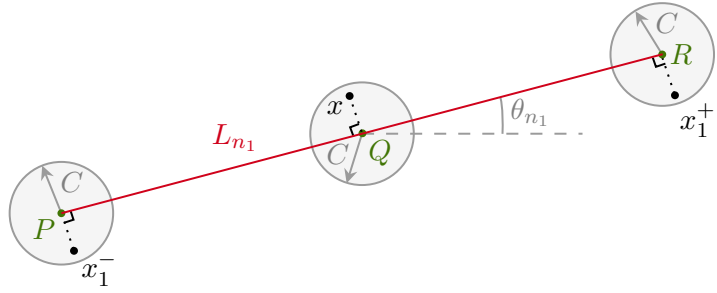

\begin{align*}
b_1^- & = p_2\big( A(x_1^--x)\big)\\
& \le -\frac{d_{n_1}-2C}{2}\sin\theta_{n_1}+2C + |\tan \theta'_{n_2}\,p_1(x_1^--x)|\\
& \le -\frac{d_{n_1}}{2}\sin\theta_{n_1}+3C + |\tan \theta'_{n_2}||a_1^-|\\
& \le -\frac{d_{n_1}}{2}\sin\theta_{n_1}+3C + 2|\tan \theta_{n_2}|d_{n_1},
\end{align*}
where the last inequality is a consequence of \eqref{EqA1-}. Because we have supposed $|\tan \theta_{n_2}|d_{n_1}\le C$, we get
\[b_1^- \le -\frac{d_{n_1}}{2}\sin\theta_{n_1}+5C.\]
Moreover, 
\[h_{n_1}\le p_2(R-P)+2C = \sin\theta_{n_1} d(R,P)+2C \le \sin\theta_{n_1}\big(d_{n_1}+2C\big)+2C \le \sin\theta_{n_1}d_{n_1}+4C,\]
so
\[b_1^- \le -\frac{h_{n_1}-4C}{2}+5C \le -\frac{h_{n_1}}{2}+7C .\]
Hence, because $h_{n_1}\ge 20m$, $6C\le m$ and $n\le 2m$,
\begin{equation}\label{EqB1-}
b_1^-+n  \le -10m+8C +2m= 8(C-m)\le -C.
\end{equation}

Using the fact that $n\ge m\ge 5C$ and that (because $\theta_{n_1}\ge 0$) $p_2(x_1^+)\ge -2C$ (see Figure~\ref{Figxx1-}), we get
\begin{equation}\label{EqB1+}
b_1^+ + n \ge C.
\end{equation}

The estimates \eqref{EqA1-2}, \eqref{EqB1-} and \eqref{EqB1+} allow to apply the following lemma:

\begin{lemma}
Let $M,C\in\R_+$, and $\gamma_2$ be a path of $\R^2$ linking the points $(-M,0)$ and $(M,0)$ of $\R^2$, that is included in $[-M,M]\times [-C,C]$. Let $\gamma_1$ be a path of $\R^2$ linking the points $(a_1^-,b_1^-)$ and $(a_1^+,b_1^+)$ of $\R^2$, that is included in $(-M,M)\times \R$, with $b_1^-<-C$ and $b_1^+>C$. 

Then the paths $\gamma_1$ and $\gamma_2$ meet.
\end{lemma}

\begin{proof}
It suffices to define the path $\alpha_2$ by concatenating $(-\infty,-M)\times\{0\}$, $\gamma_2$ and $(M,+\infty)\times\{0\}$. This is a Jordan loop of the Alexandroff compactification $\Sp^2$ of $\R^2$, which is isotopic --- with an isotopy with support included in $[-M,M]\times [-C,C]$ --- to the path $\R\times\{0\}$. It is then easy to see that the points $(a_1^-,b_1^-)$ and $(a_1^+,b_1^+)$ lie in different connected components of $\Sp^2\setminus \alpha_2$, and hence that $\alpha_2$ and $\gamma_2$ intersect. But it is also easy to see that any intersection point cannot belong to $(-\infty,-M)\times\{0\}$ or $(M,+\infty)\times\{0\}$; this implies that $\gamma_1$ and $\gamma_2$ intersect.
\end{proof}

From this we deduce that for any $n\in[m,2m]$, the paths $\gamma_1+w_m$ and $\gamma_2$ intersect. Hence, for any $n\in [m,2m]\cap \N$, we have
\[\tilde f^{n_1}\big(D_0+(0,n)\big) \cap \tilde f^{n_2}(D)\neq\emptyset,\]
equivalently
\[\big(D_0+(0,n)\big) \cap \tilde f^{n_2-n_1}(D)\neq\emptyset.\]
This implies that there exists two points $z_1,z_2\in f^{n_2-n_1}(D)$ such that $|p_1(z_1)-p_1(z_2)|\le 1$ and $|p_2(z_1)-p_2(z_2)|\ge 3C$. This contradicts the fact that $n_2-n_1$ satisfies the two first points of Claim~\ref{LastClaim}: if $\theta_{n_2-n_1}$ is small enough such a property is incompatible with $\tilde f^{n_2-n_1}(D)-\tilde f^{n_2-n_1}(x_0)\subset B(L_{n_2-n_1}, C)$.
\end{proof}

\small

\bibliographystyle{amsalpha}
\bibliography{Biblio}

\end{document}